\documentclass{article}
\usepackage[utf8]{inputenc}
\usepackage[utf8]{inputenc}
\usepackage{amsmath,amssymb}
\usepackage[T1]{fontenc}
\usepackage[english]{babel}
\usepackage{latexsym}
\usepackage{amsfonts}
\usepackage{amsthm}
\usepackage{amssymb}
\usepackage{graphicx}
\usepackage{textcomp}
\usepackage{hyperref}
\usepackage{dsfont}
\usepackage{color}
\usepackage{url}
\usepackage{chngcntr}
\usepackage{subcaption}
\usepackage{color}
\usepackage{mathtools}
\usepackage{eurosym}
\usepackage{stmaryrd}
\mathtoolsset{showonlyrefs}

\usepackage{geometry}
\usepackage[titletoc,title]{appendix}

\counterwithin{figure}{section}
\counterwithin{table}{section}

\newtheorem{theorem}{Theorem}
\newtheorem{proposition}{Proposition}
\newtheorem{rmq}{\textsc{Remark}}

\usepackage{soul}
\usepackage{comment}
\usepackage{scrextend}
\usepackage{footmisc}
\usepackage{xcolor}
\definecolor{purple}{RGB}{128,0,128}
\definecolor{grey}{RGB}{128,128,128}
\definecolor{brown}{RGB}{150,100,0}

\newtheorem{pro}{\textsc{Proposition}}
\newtheorem{lem}{\textsc{Lemma}}

\newtheorem{coro}{\textsc{Corollary}}

\title{Unified  modelling  of epidemics by coupled  dynamics via Monte-Carlo Markov Chain algorithms.}

\author{Protin Frédéric    \thanks{corresponding author: protin@torus-actions.fr} \footnotemark[3] \footnotemark[9] 
\and  Martel Jules \thanks{Max Planck Institute for Mathematics, Vivatsgasse 7, 53111 Bonn, Germany}
\and  Nguyen Duc Thang \thanks{Torus Actions SAS, 3 Avenue Didier Daurat. Research supported by Torus Actions}
\and  Nguyen T.T. Hang \footnotemark[3]
\and  Piffault Charles  \footnotemark[3]
\and  Rodr\'iguez Willy  \thanks{Torus9 SAS, 3 Avenue Didier Daurat, 31400 Toulouse, France. Research  supported by Torus9.}
\and  Figueroa Iglesias Susely \footnotemark[3]
\and  T\^o Tat Dat \thanks{Sorbonne Universit\'e, IMJ-PRG, 75252 Paris  c\'edex 05, France}
\and  Tuschmann Wilderich  \thanks{Fakultät für Mathematik, Karlsruher Institut für Technologie (KIT), Englerstr. 2, D-76131 Karlsruhe, Germany. The research
was supported by the HeKKSaGOn German–Japanese University Network Research Project “Mathematics at the Interface of Science and Technology”}
\and  H\^ong V\^an  L\^e \thanks{Institute of Mathematics of the Czech Academy of Sciences, Zitna 25, 11567 Praha 1, Czech Republic. The research
was supported by  GA\v CR-project 18-01953J and	 RVO: 67985840}
\and Yeo  Ténan  \thanks{UFR Mathématiques et Informatique, Université Félix Houphouët Boigny, Abidjan, Côte d'Ivoire}
\and  Nguyen Tien Zung \thanks{Institut de Mathematiques de Toulouse,  Université Toulouse 3, 18 Route de Narbonne, 31400 Toulouse, France.}}

\begin{document}

\maketitle









\

\bigskip


\bigskip

\abstract{
To forecast the time dynamics of an epidemic,
we propose a discrete stochastic model  that unifies and generalizes previous approaches to the subject. Viewing a given population of individuals or groups of individuals with given health state attributes as living in and moving between the nodes of a graph, we use Monte-Carlo Markov Chain techniques to simulate the movements and health state changes of the individuals according to given probabilities of stay that have been preassigned to each of the nodes. 
We utilize this model to either capture and predict the future geographic evolution of an epidemic in time, or the evolution of an epidemic inside a heterogeneous population which is divided into homogeneous sub-populations, or, more generally, its evolution in a combination or superposition of the previous two contexts.
We also prove that when the size of the population increases and a natural hypothesis is satisfied, the stochastic process associated to our model converges to a deterministic process. Indeed, when the length of the time step used in the discrete model converges to zero, in the limit this deterministic process is driven by a differential equation yielding the evolution of the expectation value of the number of infected as a function of time.
In the second part of the paper, we apply our model to study the evolution of the Covid-19 epidemic. We deduce a decomposition of the function yielding the number of infectious individuals into "wavelets", which allows to trace in time the expectation value for the number of infections inside each sub-population. Within this framework, we also discuss possible causes for the occurrence of multiple epidemiological waves.\\

}


\section{Introduction}

The Covid-19 pandemic  has given rise to manifold new studies which aim at understanding the dynamics of the disease spread and predicting its future evolution (see, for instance, \cite{1,2,3,4}). 

In a large population context, several models have been proposed to grasp the phenomena of multiple front epidemic waves. These multiple waves have for example been explained by a diversity of sub-populations in the same locality, with their own characteristics, spreading the disease from one to the other by interacting.  The first type of model is for instance described in \cite{KN}, where a Forced-SIR model is used to capture the multiple waves, interpreted as subepidemics, underlying a country’s overall incidence curve.\\
Other models rather consider the succession of contamination episodes between geographic areas. For instance, in the seminal work \cite{BH13}, the authors consider a network model for a pandemic. Each node of the graph is a town, which is driven by a basic SIR model. The edges represent airline connections between towns. The population is mixed by a Markov process, superimposed to the SIR dynamics locally defined for each town. The authors define a pseudo-distance between two towns from the matrix of the Markov process. They show that, for the SARS epidemic of 2003, and for the 2009 H1N1 influenza pandemic, the pseudo-distance from a town to the origin of the pandemic is strongly correlated to the time of contamination of this town. In \cite{BC} this model is applied to a graph of towns in Mexico, where the order of contamination of these towns for the Covid-19 epidemic is retrieved.\\
Some models represent in a common framework the heterogeneity of the population due to subdivisions into sub-populations or to geographical dispersion. Thus in \cite{Chowell}, the authors introduce an epidemic model composed of overlapping sub-epidemic  waves, each representing the dynamics of groups in the population. They argue that these groups are determined by spatially clustered pockets, population mobility patterns, infections moving across different risk groups, and so on. We refer to \cite{Torus} for a generalization of sub-epidemic waves using wavelets approach. In particular this latter method gives very reasonable forecasting. We adopt this point of view, and, in a context of large populations, we consider a population moving on a graph, whose nodes can represent geographical zones as towns or countries swapping travellers, and sub-populations more or less close in contacts.\\
Modeling the spreading of an epidemic in a small population requires a different treatment, as the stochastic effects become important. Here \cite{CMM} proposes a discrete SIR model to estimate the evolution of the expectation of the number of infected individuals for measle epidemic. In the same way as in in \cite{CMM}, 
T. Yeo and his colleagues \cite{nz} give both a stochastic and deterministic model of an epidemic spreading.  
The authors describe a spatial model for the spread of a disease on a grid of a bounded domain. The stochastic model consists of a random Markov epidemic model as a Poisson process driven stochastic differential equation.  They prove that  the stochastic model converges to the corresponding deterministic patch model, as the size of the population tends to infinity. On the other hand, they show that the stochastic model converges to a diffusion SIR model as  the size of the population tends to infinity and the mesh of the grid goes to zero.\\
We propose here to bring the contexts of small and large populations together in a common framework. First we present a discrete stochastic model of an epidemic spreading on a graph. It uses a Monte-Carlo Markov Chain method for simulating the displacement of individuals according to given probabilities for the individuals to be in specific node at certain time step. Then we prove that when the length of the discrete time step converges to zero, this process is driven by differential equations giving the evolution of the expectation value of the number of infected and susceptible individuals as a function of time. We interpret the  resulting differential equations as classical local SIR dynamics, coupled by a diffusion operator, associated to the diffusion of the virus by the interaction of individuals. In a large population context, nodes of the graph represent geographical entities or sub-populations. It is shown that, when the size of the population increases, under a natural hypothesis, the stochastic process associated to this model converges to a deterministic process. Thus this model can be used to grasp either the geographic evolution of an epidemic, or the evolution of an epidemic in a heterogeneous population, divided into homogeneous sub-populations, or, more generally, a mix of these two contexts.\\
In a small population context, the differential equations system driving the expectation of the number of infected and susceptible individuals is still valid, but the diffusion operator has a slightly different interpretation. It expresses the occupation time of agents in each node interpreted as a geographical zone, such as a room in a building or a ship, according to a statistical schedule. \\

Section \ref{sec:theory} contains the description of the model and several theoretical derivations. 
After motivating the model in Subsection \ref{sub:intro}, and describing the Monte-Carlo Markov chain (MCMC) algorithm adapted to our context in Subsection \ref{sub:MCMC}, the model in abstract form is presented in Section \ref{sec:onenode}
 in the case where the graph is reduced to one node, for simplicity  and first derivations. In particular, it is shown that under a natural hypothesis, when the size of the population increases, the stochastic process converges towards a deterministic process, in a mathematically precise sense (Theorems \ref{martingale}, \ref{MeanField}). 
 
We explain how to generalize these results in order to apply them to variations of our model, which, in particular, include SIR models taking into account saturation effects in the contamination process, as proposed for instance in \cite{ZJ} in the time-continuous case. An estimation of the rate of convergence is also given. 
In response to the obvious stability question whether our proposed discrete approach also does allow for an interpretation and version in the setting of a continuous time parameter, in subsection \ref{sub:exp} we show that when the lengths of the time steps on which our model is based converge to zero, the dynamics of the average number of infected and susceptible individuals can actually be expressed by a differential equation (Equation \eqref{sys1}). In addition, we prove there that this is true as well if we allow for the more general setting of an arbitrarily chosen, though, of course,  fixed, time delay that will naturally occur when infected individuals pass on to 'recovered'' state.\\
We also propose another method of prediction which is based upon iterating calculations of conditional expectations. We anticipate that this approach can be successfully transferred to the description and study of models of yet higher order of complexity, for which a differential equation reflecting the average evolution of the epidemic would be rather difficult, if not impossible, to obtain.\\
In Subsection \ref{sub:var} we prove that the variances of the number of infected, susceptible and recovered people divided by the size of the population converge towards $0$ when this size increases and the time step decrease. A precise upper bound is given for the case of recovered people (Proposition \ref{maj_var}), according to a method that can be adapted to the case of infected and susceptible individuals.
In Section \ref{sec:several}, the model is presented in generality for a graph with several nodes. Depending on the context, a node can represent a room in a building or a ship, a geographical zone, or a sub-population. We explain how the results of Section \ref{sec:theory} transfer to this case. We also provide a heuristic interpretation of how the diffusion effects slow down the spread of an epidemic. A model for several groups, making sense in a context of small population, is also presented in Subsection \ref{several_group}.\\
Numerical experiments are presented in Section \ref{sec:num}. In particular we present a fitting of the model by data from the Covid-19 in Singapore. We also show how our model allows to differentiate the effects of lockdown or social distancing measures from the effects of limiting movement measures on a large scale, imposed by, say, cancellation of plane flights or a ban on changing geographic regions for large numbers of individuals.

\medskip
\noindent
{\bf Related works:}  Simulations using network models have been used  as an effective tool to study the properties of epidemics, see for example \cite{Bart} and \cite{Volz}.

Recently some authors have  implemented network or agent-based models to simulate and study the spread of Covid-19 through various populations. In \cite{SM} the authors introduced a method for modeling disease transmission dynamics through a SEIR model which relies on a contact network between individuals rather than cities or large populations. In \cite{Pra} the authors used a network model to study the spread of Covid-19, in particular  this contact network was between cities rather than individuals. 
 In our present work, we introduce a common framework containing both  contexts of small and large populations. 

We also refer  \cite{Hoertel}  where the authors used a network-based model to study the efficacy of various  interventions in France, including national lockdown, mask wearing, distancing measures. They also studied the impacts of these measures on ICU bed occupancy, numbers of total cases, and numbers of total deaths.

\section{Epidemic on a graph with MCMC method}\label{sec:theory}

\subsection{Introduction of the method}\label{sub:intro}

We present here a simulation intended to motivate our model in a small population context.

The MCMC method (Monte-Carlo Markov Chain) is traditionally used to simulate a given law on a very large state space. More precisely, the method is based on the simulation of a Markov chain converging to this law when the number of iterations tends to infinity. We refer for example to \cite{Robert} for details. In the present study, we propose to use the MCMC algorithm for a different purpose, namely to simulate a random walk on a given graph, with prescribed limiting probabilities of being in each node. 

For this simulation we describe the process as follows. Two populations of 100 individuals each, are moving between the rooms represented by nodes in a graph (Figure \ref{fig:three.png}). Each individual is an object in the sense of OOP (Object-oriented programming), with the following attributes: state (S, I, or R, for Susceptible, Infectious or Recovered, respectively), duration of infection for the individuals in the state $I$, the number of the room where the individual is located, and his group G1 or G2, as defined hereafter.\\

The individuals of the group G1 move mainly between the rooms 1, 2, 3, 4, while the individuals of the group G2 move between the rooms 4, 5, 6. The two groups meet therefore in room 4. For each of the two groups, we can choose the probabilities for an individual of this group to be in each room. For the group G1, we have taken in this example the probabilities 0.28, 0.28, 0.28, 0.16  for each individual to be in room 1, 2, 3, 4 respectively, and for the group G2, we have taken probabilities 0.2, 0.4, 0.4 for each individual to be in room 4, 5, 6 respectively. Thus each individual is two times less often in room 4, where the groups meet, than in the other rooms that he frequents. The movement of the individuals is then performed independently according to a Markov chain simulated by MCMC method, having these given limiting probabilities to be in each room. 

The contamination process is simulated as follows. For each susceptible individual at a given node, the probability that it goes to state $I$, i.e. it becomes infected, is 
\begin{equation}
\label{proba}
p_I=1-e^{-0.01*n_I},
\end{equation}
where $n_I$ is the number of Infected individuals at that node, who remains Infected for more than 5 steps (to represent a latency time). Note that this choice is equivalent to suppose the infection of a Susceptible with probability $1-e^{-0.01}$ by going through the Infected successively. It is also equivalent to transform a proportion of susceptible into infected according to a binomial law with parameters $n_S$, the number of susceptible individuals, 
and $1-e^{-0.01*n_I}$. Furthermore, each Infected can pass to Recovered state either with a probability of 0.05, or after the infection time exceeds 20 steps, and it applies to all nodes.
The simulation for this process is shown in Figure \ref{fig:two.png}. \

\begin{figure}[!h]
     \centering
     \begin{subfigure}[b]{0.35\textwidth}
         \centering
         \includegraphics[width=3.5cm, height=4cm]{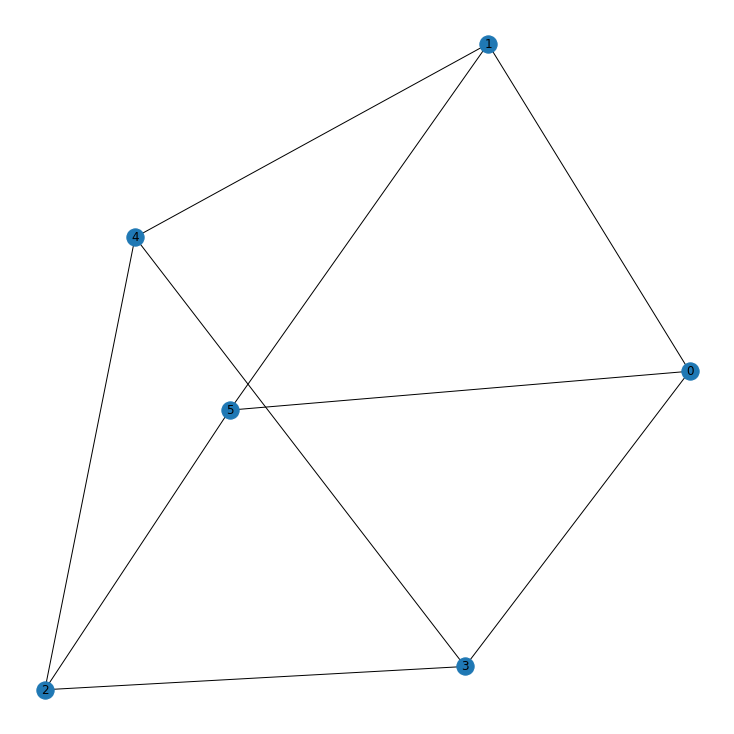}
         \caption{The 6 rooms of the simulation by the MCMC technique.}
         \label{fig:three.png}
     \end{subfigure}
     \hfill
     \begin{subfigure}[b]{0.55\textwidth}
         \centering
         \includegraphics[width=6.8cm, height=4.8cm]{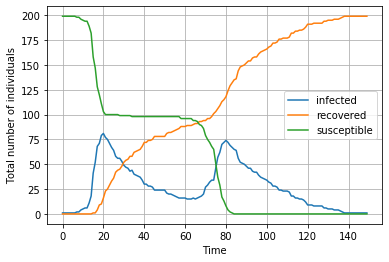}
         \caption{Evolution of an epidemic simulated by the MCMC technique.}
         \label{fig:two.png}
     \end{subfigure}
      \caption{Rooms of the simulations and evolution of an epidemic
simulated by the MCMC technique. We refer to the text for the values of the parameters.}
\end{figure}

The two peaks in Figure \ref{fig:two.png} are robust, in the sense that they appear in almost every simulation. This shape seems to depend more on the topology of the graph, the choice of the parameters, and the rules of attendance of the rooms by each group, than on pure luck. With the same conditions as above, now suppose that groups 1 and 2 meet at nodes 2, 3 and 4, and no longer only at node 4. For the group 1, we take probabilities 0.17, 0.17, 0.49, 0.17 for each individual to be in room 1, 2, 3, 4 respectively, and for the group 2, we take probabilities  0.14, 0.44, 0.14, 0.14, 0.14 for each individual to be in room 2, 3, 4, 5, 6 respectively. We show this simulation in Figure \ref{fig:one.png}. The groups are thus more mixed than previously, and we note that there is now only one peak (Figure \ref{fig:one.png}).

\begin{figure}[!h]
    \centering
    \includegraphics[width=0.6\textwidth]{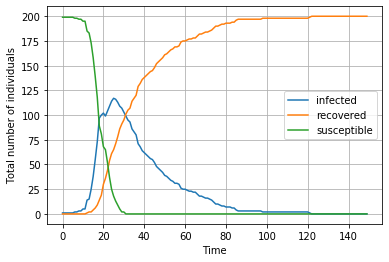}
    \caption{Evolution of an epidemic simulated by the MCMC technique, other simulation. We refer to the text for the values of the parameters.}
    \label{fig:one.png}
\end{figure}

\paragraph{Reasoning for \eqref{proba}.} We give here some heuristic reasons for the choice of \eqref{proba} in the previous model. Recall that, in a given room, an infectious individual contaminates a susceptible one according to the value of a random variable with a Bernoulli distribution. We wish to estimate the parameter of this distribution, which depends on the number of infectious individuals in the room. Suppose that the process evolves in continuous time. It can be expected that the time $T$ before a given susceptible individual becomes infected by a given infectious individual follows an exponential distribution, that is, $\mathbb{P}(T<t)=1-e^{-\lambda t}$, for all $t > 0$, for some parameter $\lambda>0$. In this case, the number of infectious individuals for each unit of time follows a Poisson distribution with parameter $\lambda$. 

Now suppose that a number $n_I$ of infectious individuals can contaminate a susceptible individual $\mathcal{S}$. For every infectious individual $i=1,\ldots,n_I$, let us denote by $T_i$ the time before this individual infects $\mathcal{S}$. Then the contamination of $\mathcal{S}$ occurs in a time step $\Delta t$ if $\displaystyle \min_{i=1,..,n_I}T_i\leq \Delta t$. Since the $T_i$s are supposed to be independent,
$\displaystyle \min_{i=1,..,n_I}T_i$ follows an exponential distribution with parameter $n_I\lambda$. Indeed, for all $t>0$, 

$$\displaystyle \mathbb{P}\left(\min_{i=1,..,n_I}T_i>\Delta t\right)=\prod_{i=1,..,n_I}\mathbb{P}\left(T_i>\Delta t\right)=e^{-n_I\lambda \Delta t}.$$

\noindent Thus we retrieve \eqref{proba}: the probability of infection in a given unit of time for each susceptible in a room is $p_{\lambda I}=1-e^{-\lambda n_I}$, for some parameter $\lambda$. \\

Finally, it is straightforward to see that \eqref{proba} is invariant by a change of time scale, which is an additional motivation for this choice.

\begin{rmq}
We have supposed so far that the infection time $T$ takes values in $\mathbb{R}^+$. It is straightforward to check that if $T$ follows an exponential distribution with parameter ${\lambda n_I}t$, then the integer part $\lfloor T\rfloor$ follows a geometric distribution with parameter $1-e^{-\lambda {n_I t}}$. We then find \eqref{proba2} once again by reasoning in discrete time.
\end{rmq}

Thus, a first idea is to consider that the probability for a susceptible to be infected in a room before time $t$ is \begin{equation}
\label{proba2}
p_{\lambda I}=1-e^{-\lambda n_I}.
\end{equation}

The parameter $\lambda=\lambda(N)$ depends on the fixed amount $N$, the total size of the population. For instance, in the basic differential SIR model it is inversely proportional to $N$ (see e.g. Equation (2.1) in \cite{H}).

\subsection{More details on the MCMC algorithm}\label{sub:MCMC}
In this section we aim to explain the details of the steps in the MCMC algorithm and how to use it in our context. Let $E=\{1,...,N\}$ be the set of nodes of a graph $G$, and $\pi=(\pi_1,...,\pi_N)$ the probability measure we want to simulate. In particular, we have $\pi_i\geq 0$ for all $i\in E$, and $\sum_{i\in E}\pi_i=1$. We can suppose without loss of generality that $\forall i\in E, \pi_i> 0$.  Recall that a matrix $P$ is said to be aperiodic if, for all $i\in E$, the greatest common divisor of $\{n: (P^n)_{i,i}>0\}$ is equal to $1$. This is equivalent to the condition $(P^n)_{i,i}>0$ for all $i\in E$ and $n\in \mathbb{N}$ large enough. The following important result is well known (see e.g. Chapter 18 in \cite{K}).

\begin{theorem}\label{invariant}
Let $P$ be a transition matrix for which $\pi$ is invariant, i.e. $\pi P = \pi$, and let $(X_n)_{n\in\mathbb{N}}$ be a Markov chain on $E$ with transition matrix $P$. If $P$ is irreducible and aperiodic, then $(X_n)_{n\in\mathbb{N}}$ converges in law towards $\pi$ when $n \to \infty$, for any initial distribution $X_0$.
\end{theorem}

\noindent In fact, we have the following sufficient condition: 

\begin{proposition}
\label{prop:P_rever_P_inv}
If $P$ is reversible for $\pi$, i.e. if $\pi_iP_{ij}=\pi_jP_{ji}$ for all $i, j \in E$, then $\pi$ is invariant for $P$.
\end{proposition}

\begin{proof}(\textit{of Proposition \ref{prop:P_rever_P_inv}})\\
\noindent The statement follows from the fact that summing over $i\in E$ the equation $\pi_iP_{ij}=\pi_jP_{ji}$ gives $\pi P = \pi$. \\
\end{proof}

\noindent Let $B$ be some irreducible transition matrix on $E$ such that 
\begin{equation}
\label{sym}
B_{ij}\neq 0 \text{ if and only if } B_{ji}\neq 0.
\end{equation}
Suppose also that $B$ is not reversible for $\pi$. Then, we can construct a matrix $Q'$ from $B$ as follows, for all $i\in E$:
\begin{equation}
\label{rev}
\left\lbrace
\begin{array}{r c l}
\displaystyle Q'_{ij}&=&\displaystyle B_{ij}\min\left(1, \frac{\pi_j B_{ji}}{\pi_i B_{ij}}\right), \text{ } \forall \  j\neq i,\\
& & \\
\displaystyle Q'_{ii}&=&\displaystyle 1 - \sum_{j\neq i}Q'_{ij}.
\end{array}
\right.
\end{equation} 
Thus, it is straightforward to see that the matrix $Q'_{ij}$ is reversible for $\pi$. We deduce the following algorithm that starts by taking $B$ as the normalized adjacency matrix of a connected graph, such that \eqref{sym} is satisfied, and returns a transition matrix $Q$ which is irreducible, reversible for $\pi$ and aperiodic. Note that $B$ is irreducible since the graph is connected. 

\begin{itemize}
\item Choose a connected graph on $E$. 
\item Let $A$ be the adjacency matrix of this graph, normalized such that the sum of each row equals $1$.
\item In the case where $A$ is already reversible for $\pi$, it is not necessary to use the construction in \eqref{rev}, and we take $Q'=A$ for the transition matrix. If $A$ is not reversible for $\pi$, then we take $Q'$ given by \eqref{rev} from $A$. It remains to possibly modify the matrix $Q'$ to make it aperiodic, which is the subject of the next step.
\item If $A$ is aperiodic, then return $Q=Q'$. Otherwise, we can slightly modify $Q'$ in order to become aperiodic and remain irreducible and reversible, for example by returning $Q = \varepsilon \ Id + (1- \varepsilon) Q'$, for some small $\varepsilon > 0$, where $Id$ is the identity matrix of the same order as $Q'$.

\end{itemize}

{\it{We suppose here and in the sequel that the graph is non-oriented, in the sense that the matrix $Q$ is symmetric.}} This corresponds to an unbiased diffusion of the agents in the graph. Equation \eqref{rev} is then simplified. Moreover, by adding arrows at each node to itself (see the next paragraph for details), the matrix $Q$ becomes aperiodic, so the last point in the previous algorithm is removed.

\subsubsection{Control of the speed of diffusion.} 
When we use the algorithm given above to construct the matrix associated to the MCMC approach from an adjacency matrix of a graph and a given probability law on its nodes, we can choose a non-zero value for the diagonal elements of the adjacency matrix. From the point of view of the graph, this comes down to authorizing the passage of a node towards itself. An agent is considered to be stationary during a time step, when he takes such a route from a room to itself. Thus, multiplying such passages by increasing the diagonal values of the previous adjacency matrix, allows us to slow down the speed of diffusion from a room to another one.\\
More precisely, let $Q_n$ be a matrix constructed as previously indicated, that is the reversible transition matrix of a Markov process on a graph such as the frequency of occupation of each node has been prescribed. Here $n$ is the common multiplicity of the nodes of the graph, considered s a parameter, while the elements in $Q_n$ out of the diagonal are fixed. Then for a fixed probability vector $X$, we verify
\begin{equation}
\label{MCMCdiffusion}
\displaystyle X^T Q_n =X - \frac{1}{n}X^T\Delta +\tilde{O}\left(\frac{1}{n^2}\right),
\end{equation} 
where \begin{equation}\label{def_laplacian}\Delta := Id-Q_1\end{equation} is the Laplacian matrix associated to $Q_1$, the matrix $Id$ being the identity matrix, and $n$ denotes the value of the diagonal elements of the adjacency matrix of the graph, considered as a parameter. Here $\tilde{O}(\cdot)$ denotes a matrix of functions $O(\cdot)$ that does not depend on $n$. Indeed, we deduce from the algorithm presented in Subsection \ref{sub:MCMC}:
\begin{equation}\label{eqdiffusion}\displaystyle Q_n=\left(1-\frac{1}{n}\right)Id+\frac{1}{n}Q_1+\tilde{O}\left(\frac{1}{n^2}\right)=Id-\frac{1}{n}\Delta+\tilde{O}\left(\frac{1}{n^2}\right).\end{equation}
Thus $\frac{1}{n}$ appear in \eqref{MCMCdiffusion} as a coefficient of diffusion. Note that $\frac{1}{n+1}$ is also the probability for each time step that an individuals  goes from one node to another.\\


\subsubsection{Large or small populations.} The model described in this subsection is suitable for small populations or large populations, but the probability vector $\pi$ appearing in Theorem \ref{invariant}, from which the matrix $Q_n$ is built according to the algorithm described in this subsection, receives in these two contexts a different interpretation. In a small population, where, for example, individuals move within a building or a ship, the vector $\pi$ prescribes the probability for each individual of being in any one of all of the given specific nodes. In the context of large populations, the nodes of the graph represent localities, or sub-populations. {\it{In the latter case, we can take as invariant vector, in order to build the matrix $Q_n$, the vector whose coordinates represent the size of the population at each node of the graph.}} Indeed, here we are looking less at the movements of individuals than at those of the virus itself. An Agent-based simulation of the model would in this case take too much time, but we will establish a differential equation depending on the matrix $Q_n$, which describes the evolution of the epidemic in continuous time.

\subsubsection{Calibration of the speed of diffusion.}\label{calibration}
We want to estimate the parameter $ n $ in order to simulate a realistic displacement. Let $\epsilon\in (0,1]$ denote the frequency of movement, that is, the average number of displacements for each individual by unit of time. Denote by $k_{step}$ the number of time steps occurring during a simulation of duration {\it{duration}}. We can estimate the probability that individual passes from a node to another during one time step by

$$
\displaystyle \frac{ duration \times \epsilon  }{k_{step}},
$$
\noindent and thus we can set 
\begin{equation}\label{eqcalibration}
\displaystyle n = \left\lfloor\frac{k_{step}}{duration \times \epsilon}\right\rfloor.
\end{equation}
Let us show how the matrix $Q_n$ constructed in subsection \ref{sub:MCMC} depends on the length of the time step $h$. Using for $h$ the expression \begin{equation}\label{dependence}\displaystyle h=\frac{duration}{k_{step}},\end{equation}
Equation \eqref{eqcalibration} gives $\displaystyle \epsilon=\frac{1}{n h }$. Thus \eqref{eqdiffusion} can be rewritten as
\begin{equation}\label{Qh}\displaystyle Q_n=Q^h:=Q_{\frac{1}{\epsilon h}}=Id-\epsilon h\Delta+\tilde{O}\left(h^2\right),\end{equation}
where $\tilde{O}(\cdot)$ denotes a matrix of functions $O(\cdot)$ that does not depend on $\epsilon$. Note that \begin{equation}\label{semigroup}\displaystyle Q^h=e^{-\epsilon h\Delta}+\tilde{O}(h^2).\end{equation}
In view of the dependence of $h$ and $k_{step}$ expressed in \eqref{dependence}, it follows that $$\displaystyle \lim_{k_{step}\rightarrow +\infty}\left(Q^h\right)^{k_{step}}=e^{-\epsilon \cdot duration\cdot \Delta}.$$Thus the discrete Markov semigroup generated by the matrices $Q^h$ converges towards a continuous Markov semigroup when the absolute value of the length of the time step converges towards $0$.

\begin{rmq}
We way interpret the formula \eqref{semigroup} as follows. For a fixed individual, we assume that the time before he passes from one node to an adjacent node follows an exponential law with parameter $\frac{1}{n}$. Then for small times the number of individuals passing from a room to the next one follows a Poisson process with parameter $\frac{1}{n}$. Hence the discrete model is embedded into a continuous time model by using $P_1^{X(t)}$ in \eqref{MCMCdiffusion}, where $X(t)$ is a Poisson process with parameter $\frac{1}{n}$. We then conclude that our equation \eqref{semigroup} corresponds to (8.3.7) in \cite{mack}. 
\end{rmq}

\section{The process for a graph reduced to one node}
\label{sec:onenode}

In this section, the model in its abstract form is presented. It is also discussed and analysed in the case where the graph is reduced to one node. The case of several nodes requires using the MCMC algorithm presented in Subsection \ref{sub:MCMC}, and will be presented in Section \ref{sec:several}.

\subsection{Description of the model}
\label{sub:stud}
In regard to the ideas introduced in Subsection \ref{sub:intro}, we state here the model in an abstract manner. We first suppose that there is one node, and look at the number of infectious individuals in this node at each step. Suppose for simplicity that the incubation time is null. Let $S^h_t$ and $I^h_t$ be the number of susceptible and infectious individuals at time $t$, where $h\in\mathbb{R}^{+*}$ is the time step. For simplicity we also do not indicate explicitly the dependence of the variables $S^h_t$ and $I^h_t$ on $N$.\\
Let $(X_i)_{i\in\mathbb{N}}$ be a sequence of random variables taking values in $\{0,1\}$. We suppose that the conditional distribution of $X_i$ given $I_t$ is Bernoulli with parameter $1-e^{-\lambda {I_t} h}$, for some $h > 0$. This is the formula \eqref{proba2}.\\
The interpretation of the variables $X_i$ is as follows. It may be assumed that for each time step the set of susceptible or infectious individuals is ordered. For all $i\in\mathbb{N}$, $X_i=0$ if the i-th susceptible at time $t$ remains in this state at time $t + h$, and $X_i=1$ if it becomes infected (and infectious).\\
Denote by $(Y_i)_{i\in\mathbb{N}}$ a sequence of random variables with Bernoulli distribution, and parameter
\begin{equation}
\label{gamma}
1-e^{-\gamma h},
\end{equation}
for some fixed $\gamma>0$. Each $Y_i$ is supposed to be independent of $I_t$ for all $t\geq 0$. The interpretation of the variables $Y_i$ is as follows. $Y_i$ will take the value $1$ each time that the i-th Infected at time $t$ becomes Recovered at time $t + h$. 
If $h$ is chosen much smaller than the mean recovering time, we can write

\begin{equation}\label{infection_}
\left\lbrace
\begin{array}{r c l}
\displaystyle I^h_{t+h}& = &\displaystyle I^h_t+\sum_{i=1}^{S^h_t}X_i-\sum_{i=1}^{I^h_t}Y_i\\\\
\displaystyle S^h_{t+h} &=&\displaystyle S^h_t-\sum_{i=1}^{S^h_t}X_i
\end{array}
\right.
\end{equation}


Note that for all $t=kh$, $k\in\mathbb{N}$, $h\in\mathbb{R}^{*+}$, the variables $I^h_t$ and $S^h_t$ depend on the time step $h$, but for convenience, we do not show this dependency in the notation, except when we want to underline it, in which case we will write $I^h_t$ and $S^h_t$. Note that the dependence on $h$ is specified in Corollary \ref{ind_h}. I would emphasize here that by construction, according to \eqref{infection_}, the sequence $\displaystyle \left(I^h_{k h},S^h_{k h}\right)_{k\in\mathbb{N}}$ is a homogeneous Markov chain.\\

\noindent {\bf{Notations.}} In the sequel, we denote by $N=I^h_t+S^h_t+R^h_t$ the total number of individuals in the population, which is independent of $t$. Here $R^h_t$ is the number of recovered individuals, recursively defined by \begin{equation}
\label{rt}
\displaystyle R^h_{t+h}=R^h_t+\sum_{i=1}^{I^h_t}Y_i.
\end{equation}
The notation $h\mapsto O(h)$ denotes a function that converges to $0$ when $h\rightarrow 0$, whose value is unimportant and may change from one location to another, even within a line. We use also the notation $O_N(\cdot)$ when it makes sense to emphasize the dependence on $N$. \\
For $h\in \mathbb{R}^{+*}$ and $k\in\mathbb{N}$, we set $\displaystyle N_{\lambda, h}(x):=\begin{pmatrix}e^{-\gamma h} & 1-e^{- \lambda  x  h}\\
0 & e^{- \lambda  x h}\\ \end{pmatrix}$, and by $\displaystyle M_{k,\lambda,h}:=N_{\lambda,h}(I_{kh})=\begin{pmatrix}e^{-\gamma h} & 1-e^{- \lambda  I_{kh}  h}\\
0 & e^{- \lambda  I_{kh} h}\\ \end{pmatrix}$. Remark that all these matrices are inversible. Let us also define the forecast function $\text{forecast}_{\lambda,h}:\mathbb{R}^2\rightarrow \mathbb{R}^2$ by

\begin{equation}
\label{def_forecast}
\displaystyle \text{forecast}_{\lambda,h}(x,y):=
\begin{pmatrix}
x-x(1-e^{-\lambda y h}) \\ 
y+x(1-e^{-\lambda yh})-y(1-e^{-\gamma h})
\end{pmatrix}.
\end{equation}

\noindent (In order to keep notation simple, we do not indicate the dependency on the parameter $\gamma$). Let us give another expression of the function $\displaystyle \text{forecast}_{\lambda,h}$. Taking the conditional expectation in \eqref{infection_} with respect to variables $S^h_t$, $I^h_t$, for some $t\in\mathbb{N}h$, gives

\begin{equation}
\label{forecast_}
\displaystyle \mathbb{E}\left(I^h_{t+h}|I^h_t, S^h_t\right)=I^h_t+S^h_t\left(1-e^{-\lambda I^h_t h}\right)-I^h_t\left(1-e^{-\gamma  h}\right).
\end{equation}

\noindent Computing by the same way $\mathbb{E}\left(S^h_{t+h}|I^h_t, S^h_t\right)$, we obtain
\begin{equation}
\label{matrix_forecast}
\displaystyle \text{forecast}_{\lambda,h}(I^h_{kh}, S^h_{kh}) =
\begin{pmatrix}
\mathbb{E}(I^h_{(k+1)h}\mid I^h_{kh}, S^h_{kh}) \\ 
\mathbb{E}(S^h_{(k+1)h}\mid I^h_{kh}, S^h_{kh}) 
\end{pmatrix}=M_{k,\lambda,h}
\begin{pmatrix} I^h_{kh} \\ 
S^h_{kh}
\end{pmatrix},
\end{equation}
for all $k\in \mathbb{N}$ and $\mathbb{E}(\cdot)$ the usual expectation. Note that a rescaling can be done as follows:
\begin{equation}
\label{forecast_N}
\displaystyle \text{forecast}_{\lambda,h}\left(x,y\right) = N\text{forecast}_{N\lambda,h}\left(\frac{x}{N},
\frac{y}{N}\right). 
\end{equation}

\subsection{Convergence of the process.}\label{sub:converge1} 
In this subsection we continue to consider the abstract setting in Subsection \ref{sub:stud}. We will establish that the process normalized by the amount of the population $N$ converges towards a deterministic process when $N$ grows, and certain conditions are given. Let us state these conditions. The first one reads 
\begin{equation}
\label{growth}
0\leq \lambda(N)\leq \frac{\lambda'}{N},
\end{equation} 

\noindent for some constant $\lambda'>0$. This inequality becomes an equality in the basic differential SIR model (see e.g. Equation (2.1) in \cite{H}).\\
The second condition, also satisfied by the basic SIR differential model (see e.g. Equation (2.1) in \cite{H}), is:

\begin{equation}\label{Condition2}\text{The quantity }{N}\lambda(N)\text{ converges when }N\rightarrow +\infty.\end{equation}

\noindent The convergence of the process under Condition \eqref{growth} is specified by Theorem \ref{martingale} below. We will use the following Lemma. For $r\in\mathbb{N}$, let us denote by $\mathcal{F}^h_r$ the $\sigma$-algebras generated by the variables $S^h_{kh}$ and $I^h_{kh}$ for $k\in\mathbb{N}$, $k\leq r$.

\begin{lem}\label{lem_martingale}
\noindent The sequence $\displaystyle \left(\prod_{i=1}^{k-1} M^{-1}_{i,\lambda,h} \begin{pmatrix} I^h_{kh} \\ S^h_{kh}\end{pmatrix}\right)_{k\geq n}$ is a martingale for the filtration $\left(\mathcal{F}^h_k\right)_k$. (We use the convention that $\displaystyle \prod_{i=1}^{0} M^{-1}_{i,\lambda,h}=Id$).\\


\end{lem}

\begin{proof}Indeed, we have for $k\geq n+1$

$$\displaystyle \mathbb{E}\left(\prod_{i=1}^{k-1} M^{-1}_{i,\lambda,h} \begin{pmatrix} I^h_{kh} \\ S^h_{kh}\end{pmatrix}\mid \mathcal{F}^h_{k-1}\right) =  \prod_{i=1}^{k-1} M^{-1}_{i,\lambda,h}\mathbb{E}\left( \begin{pmatrix} I^h_{kh} \\ S^h_{kh}\end{pmatrix}\mid \mathcal{F}^h_{k-1}\right)$$

$$\displaystyle =  \left(\prod_{i=1}^{k-1} M^{-1}_{i,h}\right) M_{k-1,\lambda, h} \begin{pmatrix} I^h_{(k-1)h} \\ S^h_{(k-1)h}\end{pmatrix} = \prod_{i=1}^{k-2} M^{-1}_{i,\lambda,h} \begin{pmatrix} I^h_{(k-1)h} \\ S^h_{(k-1)h}\end{pmatrix}.$$
\end{proof}


\begin{rmq}
The sequence $\displaystyle \left(\prod_{i=1}^{k-1} M^{-1}_{i,\lambda,h} \begin{pmatrix} I^h_{kh} \\ S^h_{kh}\end{pmatrix}\right)_{k\in \mathbb{N}}$ being a positive martingale according to Lemma \ref{lem_martingale}, converges a.s. towards a couple $ \begin{pmatrix} U_h \\ V_h\end{pmatrix}$ of random variables. Moreover, for all $k\in \mathbb{N}$, we have
$$
\displaystyle \mathbb{E}\left( \begin{pmatrix} U_h \\ V_h\end{pmatrix} \mid \mathcal{F}^h_k\right)=\prod_{i=1}^{k-1} M^{-1}_{i,\lambda,h} \begin{pmatrix}
I^h_{kh} \\ 
S^h_{kh}
\end{pmatrix}.
$$ 
\end{rmq}



\noindent We can now establish:

\begin{theorem}\label{martingale}

Suppose that Condition \eqref{growth} is verified. Then for all $t\in\mathbb{N}h$ the relative number of infectious individuals $\frac{I^h_t}{N}$
becomes in the following manner more and more closer to $\frac{\mathbb{E}(I^h_t)}{N}$ as $N$ becomes larger: There exists a constant $K>0$ independent of $N$, $\forall n \in \mathbb{N}$, $\forall \epsilon > 0$, 
$$
\displaystyle \mathbb{P}\left(\left\lvert \frac{I^h_{n h}}{N} - \mathbb{E}\left(\frac{I^h_{n h}}{N}\right)\right\rvert >\epsilon\right)\leq e^{-\frac{\epsilon^2\sqrt{N}}{2(n+1)K^{n}}}+ n e^{-2\frac{\sqrt{N}}{K^{n}}},
$$ 
\noindent and 
$$
\mathbb{P}\left(\left\lvert \frac{S^h_{n h}}{N} - \mathbb{E}\left(\frac{S^h_{n h}}{N}\right)\right\rvert >\epsilon\right)\leq e^{-\frac{\epsilon^2\sqrt{N}}{2(n+1)K^{n}}}+ n e^{-2\frac{\sqrt{N}}{K^{n}}}.
$$

\end{theorem}

\begin{proof}

Denote by $P_0:\mathbb{R}^2\rightarrow \mathbb{R}$ the projection to the first coordinate, i.e. $(x,y)\mapsto x,\ \forall (x,y)\in \mathbb{R}^2$, and let $P_1:\mathbb{R}^2\rightarrow \mathbb{R}$ the function that returns the second coordinate. Fix $h>0$, and denote by $$\displaystyle A_{k}^h:=\prod_{i=1}^{k-1} M^{-1}_{i,\lambda,h} \begin{pmatrix} I^h_{kh} \\ S^h_{kh}\end{pmatrix} $$

 
\noindent for $1\leq n\leq k$, with the convention that $\displaystyle \prod_{i=1}^{0} M^{-1}_{i,\lambda,h}=Id$. Then Lemma \ref{uniform} implies that there exists a constant $C>0$ independent of $N$ such that
$$\displaystyle \frac{1}{N}\left|P_0A_{n+1}^h-P_0A_{n}^h\right|\leq \frac{C^{n}}{N}\left|P_0\prod_{i=1}^{n} M_{i,\lambda ,h}A_{n+1}^h-P_0\prod_{i=1}^{n} M_{i,\lambda ,h}A_{n}^h\right|=\frac{C^{n}}{N}\left|I_{(n+1) h}^h-\mathbb{E}\left(I_{(n+1)h}^h\mid I_{n h}^h, S_{n h}^h\right)\right|$$ $$\displaystyle = \frac{C^n}{N}\left|\sum_{i=1}^{S_{n h}^h}X_i-\sum_{i=1}^{I_{n h}^h}Y_i-\left(1-e^{-\lambda(N) I_{n h} h}\right)S_{n h}^h+\left(1-e^{-\gamma h}\right)I_{n h}^h\right|.$$

\noindent The penultimate equality follows from \eqref{matrix_forecast}, and the last equality follows from \eqref{infection_}. For $\displaystyle l,m\in[\![0,N]\!]$,  define a probability measure by $$\displaystyle\mathbb{P}^n_{l,m}(A):=\frac{\mathbb{P}(A \cap \{I_n=l,S_n=m\})}{ \mathbb{P}(\{I_n=l, S_n=m\})}.$$
Then Chernoff's bound for sums of independent Boolean variables (see e.g. Theorem 2.1 and Corollary 4.1 in \cite{Mulzer}, see also Theorem 1.1 in \cite{IK}) gives for all $\displaystyle l,m\in[\![0,N]\!]$, $\forall\epsilon>0$,

$$\displaystyle 
\begin{array}{r c l}
\displaystyle \mathbb{P}^n_{l,m}\left(\frac{1}{N}|P_0A_{n+1}^h-P_0A_{n}^h|>\epsilon\right)&\leq&\displaystyle  \mathbb{P}^n_{l,m}\left(\frac{C^n}{N}\left|\sum_{i=1}^{m}X_i-\sum_{i=1}^{l}Y_i-\left(1-e^{-\lambda(N) l h}\right)m+\left(1-e^{-\gamma h}\right)l\right|>\epsilon\right)\\
& &\\
 &\leq&\displaystyle  e^{-2\frac{\epsilon^2N}{C^{2 n}}}.
\end{array}
$$

\noindent In other words, $\forall\epsilon>0$, $$\displaystyle \mathbb{P}\left(\frac{1}{N}\left|P_0A_{n+1}^h-P_0A_{n}^h\right|>\epsilon \bigg|\mathcal{T}^h_n\right)\leq e^{-2\frac{\epsilon^2N}{C^{2 n}}}.$$ 

\noindent Here $\mathcal{T}^h_n$ denotes the $\sigma$-algebras generated by the variables $S^h_{n h}$ and $I^h_{n h}$. In particular, after taking expectation,

\begin{equation}\label{chernoff}\displaystyle \mathbb{P}\left(\frac{1}{N}\left|P_0A_{n+1}^h-P_0A_{n}^h\right|>{N^{-\frac{1}{4}}}\right)\leq  e^{-2\frac{\sqrt{N}}{C^{2 n}}}.\end{equation}

\noindent Now the sequence $\displaystyle \left(A_k^h\right)_{k\in \mathbb{N}}$ is a martingale by Lemma \ref{lem_martingale}. Then a concentration inequality for martingales (Theorem 33 in \cite{CL}, see also p.14 of this reference for the notations), together with Lemma \ref{uniform}, provided that the constant $ C> 0 $ has been chosen sufficiently large, allows us to conclude that for all $\epsilon>0$, we have
$$\displaystyle \mathbb{P}\left(\frac{1}{N}|I_{n h}^h-\mathbb{E}(I_{n h}^h)|>\epsilon\right)\leq \mathbb{P}\left(\frac{1}{N}\left|P_0A_{n}^h-\mathbb{E}(P_0A_{n}^h)\right|>\frac{\epsilon}{C^n}\right)\leq e^{-\frac{\epsilon^2\sqrt{N}}{2(n+1)C^{2 n}}}+  n e^{-2\frac{\sqrt{N}}{C^{2 n}}}.$$
\noindent We obtain by exactly the same way, $\forall\epsilon>0$,  
$$\displaystyle \mathbb{P}\left(\frac{1}{N}|S_{n h}^h-\mathbb{E}(S_{n h}^h)|>\epsilon\right)\leq \mathbb{P}\left(\frac{1}{N}\left|P_1A_{n}^h-\mathbb{E}(P_1A_{n}^h)\right|>\frac{\epsilon}{C^n}\right)\leq e^{-\frac{\epsilon^2\sqrt{N}}{2(n+1)C^{2 n}}}+ n e^{-2\frac{\sqrt{N}}{C^{2 n}}}.$$
\end{proof}

\noindent In particular, for all $t\in\mathbb{N}h$ the variables $\frac{S_t^h}{N}$ and $\frac{I_t^h}{N}$ become closer and closer to their mean as $N$ gets large. Now we will see that under Condition \eqref{Condition2} this fact can be made precise as follows: 

\begin{theorem}\label{MeanField}
Suppose that Condition \eqref{Condition2} is satisfied.  Then $\frac{S^h_t}{N}$ and $\frac{I^h_t}{N}$ both converge a.e. when $N\rightarrow +\infty$ towards a strictly positive number, provided that the initial conditions $\frac{S_0}{N}$ and $\frac{I_0}{N}$ converge when $N\rightarrow +\infty$ towards a strictly positive number.
\end{theorem}

\begin{proof}
Define for $q\in \mathbb{R}$ the matrix 
$$
\displaystyle Q_{\lambda,\gamma,h}(q):=\begin{pmatrix}
e^{-\lambda h q} & 1-e^{-\lambda h q} & 0\\
0 & e^{-\gamma h} & 1-e^{-\gamma h }\\
0 & 0 & 1
\end{pmatrix}.
$$
In the model described by \eqref{infection_}, each individual in a state $S$, $I$ or $R$ has a certain probability to change his state given by the matrix $\displaystyle Q_{\lambda N,\gamma,h}\left(\frac{I^h_t}{N}\right)$, indexed in the same order. In the terminology of \cite{MF}, the proportion of infectious individuals $\frac{I^h_t}{N}$ is the state of the system at time $t$. Noticing that, for $q\in\mathbb{R}^+$ being fixed, the matrix $Q_{\lambda N,\gamma,h}(q)$ converges uniformly in $\mathbb{R}^9$ when $N\rightarrow +\infty$ by \eqref{Condition2}, a Mean Field limit theorem (Theorem [MF:Thrm] in \cite{MF}) implies the statement.
\end{proof}

\noindent Theorem \ref{martingale} and Proposition \ref{MeanField} together imply to the following statement:

\begin{coro}\label{towardsc}
Suppose that Conditions \eqref{growth} and \eqref{Condition2} are verified. Then, for all $t\in \mathbb{N}h$, the expectation of the variables $\frac{S_t}{N}$ and  $\frac{I_t}{N}$ converge when $N\rightarrow +\infty$. \\
Moreover, these variables themselves converge in probability towards these respective limits when $N\rightarrow +\infty$.
\end{coro}

\begin{proof}
It suffices to note that the initial conditions $\frac{S_0}{N}$ and $\frac{I_0}{N}$ in the statement of Theorem \ref{MeanField} can be made arbitrarily close to $0$.
\end{proof}

\paragraph{Generalisations.}
 Now we shall propose a generalization of the results concerning the model \eqref{infection_} obtained in this subsection to the case of the following variations of this model. If the function $N\mapsto \lambda(N)$ is replaced by a time-dependent random variable $(t,N, S^h_{t}, I^h_{t})\rightarrow \lambda_t(N, S^h_{t}, I^h_{t})$, where $\lambda_t$ is measurable for all $t\in\mathbb{N}h$ and $N\in\mathbb{N}$ for the $\sigma$-algebra $\sigma(S^h_{t}, I^h_{t})$, we note that the proofs of Lemma \ref{lem_martingale} and Theorem \ref{martingale} remain valid, and thus their statements extend to this case. The only change required is to replace Condition \eqref{growth} used in Theorem \ref{martingale} by the following: $$
\displaystyle \exists \lambda'\in \mathbb{R}^{+*}, \forall t \in \mathbb{R}^+,\textbf{ }
0\leq \lambda_t(N,S^h_{t}, I^h_{t})\leq \frac{\lambda'}{N}.
$$
This generalisation includes for instance the SIR model where the incidence rate is divided by an affine function of $I^h_{t}$ for taking into account saturation effects in the contamination process, as has been proposed for instance in \cite{ZJ} in the time-continuous case.\\
In the latter case, if $\lambda_t(N, S^h_t, I^h_t)$ has the form $\lambda(N)g(S^h_t,I^h_t)$ for some continuous functions $N\mapsto \lambda(N)$ satisfying \eqref{Condition2} and $g:\mathbb{R}^2\rightarrow \mathbb{R}$, Theorem \ref{MeanField} is also still valid, since Theorem [MF:Thrm] in \cite{MF} used in its proof does still work. 


\subsection{Computation of the expected values}\label{sub:exp}

Considering  further the  abstract  setting in Subsection \ref{sub:stud}, in this subsection we will derive differential equations driving the dynamics of the expectation of the number of susceptible and infectious individuals given by \eqref{infection_} when the time step $h$ converges to $0$ (Proposition \ref{pro:sys1}). Indeed, it is expected that this passage to limit in the model is needed to represent the continuity of the real time parameter. Taking the conditional expectation with respect to variables $S^h_t$, $I^h_t$, for some $t\in\mathbb{N}h$, gives 

\begin{equation}
\label{forecast_}
\displaystyle \mathbb{E}\left(I^h_{t+h}|I^h_t, S^h_t\right)=I^h_t+S^h_t\left(1-e^{-\lambda I^h_t h}\right)-I^h_t\left(1-e^{-\gamma  h}\right).
\end{equation}

\begin{rmq}
This equality can be used from an instance of time $t$ to a future instance $t+h$, as long as $h$ is much smaller than the mean recovering time. Associated with this estimation, the equation \eqref{infection_} give the variance

\begin{equation}\label{cond_var}
\begin{array}{r l c}
V(I^h_{t+h}|I^h_t, S^h_t)&=&\displaystyle S^h_t V(X_1|S^h_t, I^h_t)+I^h_t V(Y_1|S^h_t,I^h_t)\\
& & \\
&=&\displaystyle  S^h_t (1-e^{-\lambda I^h_th}) + I^h_t(1-e^{-\gamma h}).
\end{array}\end{equation} 

\noindent We will actually see in Proposition \ref{justify_forecast_3}, Appendix \ref{appendiceB}, how to use (\ref{forecast_}) for forecasting in more distant future.
\end{rmq}

\noindent Now by taking in \eqref{forecast_} the conditional expectation with respect to the variable $I^h_t$ we obtain:

\begin{equation}\displaystyle \mathbb{E}\left(I^h_{t+h}|I^h_t\right)=I^h_t+\mathbb{E}\left(S^h_t|I^h_t\right)\left(1-e^{-\lambda I^h_t h}\right)-I^h_t\left(1-e^{-\gamma  h}\right).\end{equation}

\noindent Taking expectation, we have 


\begin{equation}\label{diff-stoc}\displaystyle \mathbb{E}\left(I^h_{t+h}\right)=\mathbb{E}\left(I^h_t\right)+\mathbb{E}\left(S^h_t\left(1-e^{-\lambda I_t h}\right)\right)-\mathbb{E}\left(I^h_t\right)\left(1-e^{-\gamma  h}\right).\end{equation}

\noindent Denote by $\displaystyle i(t):= \lim_{l\rightarrow 0^+} \mathbb{E}\left(I^l_t\right)$ and by $\displaystyle s(t):= \lim_{l\rightarrow 0^+} \mathbb{E}\left(S^l_t\right)$. These quantities are well defined thanks to Corollary \ref{limit} in Appendix \ref{appendiceB}. Moreover, the latter corollary implies 
\begin{equation}
\displaystyle i(t+h)-i(t)=\mathbb{E}\left(S^h_t\left(1-e^{-\lambda I^h_t h}\right)\right)-\mathbb{E}\left(I^h_t\right)\left(1-e^{-\gamma  h}\right)+O(h^2).
\end{equation}

\noindent Dividing by $h$ and letting $h\rightarrow 0^+$, we obtain, thanks to Corollary \ref{product_exp} in Appendix \ref{appendiceB}, and by proceeding in the same way from the second equation in \eqref{infection_}, the following result:


\begin{pro}\label{pro:sys1}

The functions $t\mapsto i(t)$ and $t\mapsto s(t)$ are differentiable and satisfy the differential equations of the basic SIR dynamics, that is to say 

\begin{equation}\label{sys1}
\left\lbrace
\begin{array}{r c l}
\displaystyle \frac{\partial i(t)}{\partial t}&= &\displaystyle {\lambda}s(t)i(t)-\gamma i(t),\\ \\

\displaystyle  \frac{\partial s(t)}{\partial t} &=&\displaystyle -{\lambda}s(t)i(t).
\end{array}
\right.
\end{equation}

\end{pro}

\noindent We thus retrieve in average the well-known basic continuous SIR dynamics (see e.g. Equation (2.1) in \cite{H}). We will see in Section \ref{sec:several} that when the graph underlying the model has several nodes, diffusion effects appear, that slow down the spread of the epidemics.


\paragraph{Introduction of a time delay.} One may wish to introduce a delay between the contamination of a susceptible and its contagiousness. This delay can be used to model, for instance, the case of individuals infected with Covid-19. We can suppose that this delay is a random variable $T$, independent of $S^h_t$ and $I^h_t$ for all $t\in\mathbb{R}^+$. Then \eqref{infection_} can be replaced by
\begin{equation}\label{infection_t}
\left\lbrace
\begin{array}{r c l}
\displaystyle I^h_{t+h}&= &\displaystyle I^h_t+\sum_{i=1}^{S^h_t}X_i-\sum_{i=1}^{I^h_{t-T}}Y_i,\\  \\
\displaystyle S^h_{t+h} &=&\displaystyle S^h_t-\sum_{i=1}^{S^h_t}X_i,
\end{array}
\right.
\end{equation}


\noindent where the conditional distribution of $X_i$ given $I^h_t$ and $T$ is Bernoulli with parameter $1-e^{-\lambda {I^h_{t-T}} h}$. Taking in \eqref{infection_} expectation according to $S^h_t, S^h_t, T$ and pursuing the computation as previously, we obtain 

\begin{equation}
\left\lbrace
\begin{array}{r c l}
\displaystyle \frac{\partial i(t)}{\partial t}&= &\displaystyle {\lambda}s(t)\left(i(t)\ast P_T\right)-\gamma i(t)\ast P_T,\\ 
\\
\displaystyle \frac{\partial s(t)}{\partial t} &=&\displaystyle -{\lambda}s(t)\left(i(t)\ast P_T\right),
\end{array}
\right.
\end{equation}


\noindent where $P_T$ denotes the measure $P_T(A):=\mathbb{P}(T\in A)$ on $\mathbb{R}$ and $\ast$ denotes the convolution product. For instance, if $T=t_0$ is constant, we have

\begin{equation}
\left\lbrace
\begin{array}{r c l}
\displaystyle \frac{\partial i(t)}{\partial t}&= &\displaystyle {\lambda}s(t)i(t-t_0)-\gamma i(t-t_0),\\ 
\\
\displaystyle \frac{\partial s(t)}{\partial t} &=&\displaystyle -{\lambda}s(t)i(t-t_0).
\end{array}
\right.
\end{equation}
Note that this system is similar to \eqref{sys1} with a time delay $t_0$ in this case. 


\subsection{Computation of the variance}\label{sub:var}

In this subsection, we will give an upper bound on the variance at each time of the number of recovered individuals $\frac{R^h_t}{N}$ in the discrete model expressed by \eqref{infection_} 
(Proposition \ref{maj_var}). Here $t=kh $ for some $k\in\mathbb{N}$. Recall that $h$ is the time step in this model, and that $N$ is the total amount of the population. This will imply that the variances of $\frac{S^h_t}{N}$, $\frac{I^h_t}{N}$ and $\frac{R^h_t}{N}$ vanish when $N\rightarrow +\infty$ and $h\rightarrow 0^+$, without taking into account the order of these latter convergences.\\

First let us recall a well-known relation between the variance $V(X)$ of a random variable $X$ with finite variance defined on a probability space $A$, and the conditional variance $V^{\mathcal{B}}(X)=\mathbb{E}^{\mathcal{B}}((X-\mathbb{E}^{\mathcal{B}}(X))^2)$
according to a $\sigma$-field $\mathcal{B}$ defined on $A$ (see e.g. \cite{Weiss} p. 385–386.)
\begin{equation}\label{variance}V(X)=V(\mathbb{E}(X|{\mathcal{B}}))+\mathbb{E}(V^{\mathcal{B}}(X)). \end{equation}
For simplicity, we denote by $V^{X}(\cdot)$ the conditional variance according to the $\sigma$-field $\sigma(X)$ generated by some random variable $X$. For further convenience, we will also use also the notation $\mathbb{E}^{X}\left(\cdot\right)$ for the conditional expectation $E\left(\cdot \mid X\right)$. We have from \eqref{rt}:\\

\noindent $$\displaystyle V^{I_t}(R^h_{t+h})=V^{I_t}(R^h_t)+e^{-\gamma h}(1-e^{-\gamma h})I_t.$$ Taking expectation, thanks to Lemma \ref{lemma} we obtain:
$$V(R^h_{t+h})-V(\mathbb{E}^{I_t}(R^h_{t+h}))=V(R^h_t)-V(\mathbb{E}^{I_t}(R^h_t)) + \mathbb{E}(I_t)(1-e^{-\gamma h})e^{-\gamma h}.$$

\noindent Thus $$\displaystyle V(R^h_{t+h})-V(R^h_t)=\mathbb{E}(I_t)(1-e^{-\gamma h})e^{-\gamma h}+V(\mathbb{E}^{I_t}(R^h_{t+h}))-V(\mathbb{E}^{I_t}(R^h_t))$$
$$\displaystyle =\mathbb{E}(I_t)(1-e^{-\gamma h})e^{-\gamma h}+V\left(\mathbb{E}^{I_t}(R^h_{t})+I_t\left(1-e^{-\gamma h}\right)\right)-V\left(\mathbb{E}^{I_t}\left(R^h_t\right)\right)$$
$$\displaystyle =\mathbb{E}(I_t)(1-e^{-\gamma h})e^{-\gamma h}+2Cov\left(\mathbb{E}^{I_t}(R^h_t), I_t(1-e^{-\gamma h})\right)+V(I_t(1-e^{-\gamma h})).$$

\noindent Thus we have the upper bound:
\begin{equation}\label{variance_lab}\displaystyle V(R^h_{t+h})-V(R^h_t) \leq \mathbb{E}(I_t)\gamma h + 2\sqrt{V(\mathbb{E}^{I_t}(R^h_t))V(I_t)}\gamma h +O_N(h^2),\end{equation}

\noindent where $O_N(h^2)\leq (1-e^{-\gamma h})^2 N^2$. Then Lemma \ref{lemma} in Appendix \ref{appendixA} leads to:
\begin{equation}\label{forgronwall}\displaystyle V(R^h_{t+h})-V(R^h_t)\leq \mathbb{E}(I_t)\gamma h + 2\sqrt{V(R^h_t)V(I_t)}\gamma h +O_N(h^2).\end{equation}

\noindent According to the fact that $V(R_0)=0$, and recalling that $t=k h$, $k\in\mathbb{N}$, we can rewrite \eqref{forgronwall} using a telescopic sum: $$\displaystyle V(R^h_{t+h})-V(R^h_t)\leq \mathbb{E}(I_t)\gamma h + 2\sqrt{V(I_t)\sum_{i=1}^{k-1} \left(V(R^h_{(i+1)h})-V(R^h_{ih})\right)}\gamma h +O_N(h^2).$$

\noindent The inequality $\displaystyle \| \cdot \|_2\leq  \| \cdot \|_1$ between $p$-norms in $\mathbb{R}^{k}$ gives the upper bound

$$\displaystyle V(R^h_{t+h})-V(R^h_t)\leq \mathbb{E}(I_t)\gamma h + 2\sqrt{V(I_t)}\gamma h\sum_{i=0}^{k-1} \sqrt{ V(R^h_{(i+1)h})-V(R^h_{ih})} +O_N(h^2).$$

\noindent To it, we can now apply a discrete Gronwall-type inequality (Consequence 1 of Theorem 105 in \cite{D}). After dividing by $N^2$ the terms of the previous inequality, we obtain, provided that $h\leq 1$:

$$\displaystyle V\left(\frac{R^h_{t+h}}{N}\right)-V\left(\frac{R^h_t}{N}\right)\leq \mathbb{E}(I_t)\frac{\gamma h}{N^2}+h^2$$ $$+2\sqrt{V\left(\frac{I_t}{N}\right)}\frac{\gamma h }{N} \sum_{s=0}^{k-1}\sqrt{{\mathbb{E}(I_{sh})\gamma h}}\prod_{i=s+1}^{k-1}\frac{\sqrt{V(I_{ih})}\gamma h }{{N}}.$$

\noindent Since $t=kh$, $V(I_t)\leq N^2$ and $\mathbb{E}(I_t)\leq N$, we obtain by supposing that $h\leq \min{\left(1,\frac{1}{\gamma}\right)}$: 

$$\displaystyle V\left(\frac{R^h_{t+h}}{N}\right)-V\left(\frac{R^h_t}{N}\right)\leq \frac{\gamma h}{N}+2\gamma^2 h \sqrt{\frac{\gamma h}{N}}{ t}+h^2.$$

\noindent Finally, recalling that $V(R_0)=0$, a telescopic sum argument gives the following statement:\\

\begin{pro}\label{maj_var}
For all $h\in\mathbb{R}^+$ such that $h\leq \min\left(1,\frac{1}{\gamma}\right)$, and all $t\in\mathbb{N}h$, we have the upper bound:

$$\displaystyle V\left(\frac{R^h_{t}}{N}\right)\leq \frac{\gamma t}{N}+\frac{2\gamma^2 t^2 }{\sqrt{N}}+th.$$

\noindent In particular, \begin{equation}\label{limN}\lim_{N,\frac{1}{h}\rightarrow +\infty} V\left(\frac{R^h_t}{N}\right)=0.\end{equation}
\end{pro}

\noindent A similar method can be used to majorize the variances $V\left(\frac{S^h_t}{N}\right)$ and $V\left(\frac{I^h_t}{N}\right)$. Let us just prove that for all $h>0$ sufficiently small, \begin{equation}\label{towardszero}\displaystyle \lim_{N,\frac{1}{h}\rightarrow +\infty} V\left(\frac{S^h_t}{N}\right)=\lim_{N,\frac{1}{h}\rightarrow +\infty}V\left(\frac{I^h_t}{N}\right)=0.\end{equation}
From $N=I^h_t+S^h_t+R^h_t$, we obtain



\begin{equation} \label{conservation2}V(R^h_t)= V(S^h_t)+V(I^h_t)+2Cov(S^h_t,I^h_t).\end{equation}

\noindent Then Corollary \ref{leqn3} or Corollary \ref{product_exp} in Appendix \ref{appendiceB}, together with \eqref{conservation2} and \eqref{limN}, gives the following fact, proving \eqref{towardszero}:
$$\displaystyle 0\leq\limsup_{N,\frac{1}{h}\rightarrow +\infty}V\left(\frac{S^h_t}{N}\right)+\liminf_{N,\frac{1}{h}\rightarrow +\infty}V\left(\frac{I^h_t}{N}\right)\leq \limsup_{N,\frac{1}{h}\rightarrow +\infty}V\left(\frac{R^h_t}{N}\right)=0,$$

$$\displaystyle 0\leq\liminf_{N,\frac{1}{h}\rightarrow +\infty}V\left(\frac{S^h_t}{N}\right)+\limsup_{N,\frac{1}{h}\rightarrow +\infty}V\left(\frac{I^h_t}{N}\right)\leq \limsup_{N,\frac{1}{h}\rightarrow +\infty}V\left(\frac{R^h_t}{N}\right)=0.$$


\section{The case of several nodes. }\label{sec:several} 
Let us now suppose that there are several nodes in the graph of the model. Here a node can represent a locus such as a room in a context of a small population, or a homogeneous sub-population or a geographical zone, in a context of a large population. Let $j=1,..,n$ be the nodes, and suppose that each individual moves from a node to another one according to a Markov chain with transition matrix $Q$. {\it{Here $Q$ is the matrix constructed by the MCMC algorithm described in Section \ref{sub:MCMC}.}} When necessary to exhibit the dependency of the matrix $Q$ on $h$, as expressed by \eqref{Qh}, we will denote it by $Q^h$ as in \eqref{Qh}. We also suppose that the infection rate depends on the node. For instance, if a node represents a room on a ship, the infection rate in this room is depending on its surface and its degree of ventilation. Thus we can assign a weight to each room, and this in such a way that e.g. a room has double weight than another one if its power of infection is the same as in the second one for a susceptible individual staying twice longer. In the same way, if a node represents a sub-population, its infection rate is depending on numerous parameters, such that its density, its mobility, etc.\\

Let ${I}^h_j(t)$ and ${S}^h_j(t)$ be the number of infected and susceptible agents at time $t$ in the node $j$. Denote by 
$$
\displaystyle \tilde{S_t} := \left(S_1(t),S_2(t),...,S_n(t)\right)^T,\text{ }\displaystyle\tilde{I} := \left(I_1(t),I_2(t),...,I_n(t)\right)^T.
$$ 
The model at each time $t$ is described from the initial number of infectious and susceptible individuals by the following recursive process:
\begin{itemize}
    \item First, the evolution of ${I}^h_j(t)$ and ${S}^h_j(t)$ in each node $j=1,\cdots,n$ from time $t$ to $t+h$ is governed by equations \eqref{infection_}. We thus obtain vectors $\hat{I}^h_{t+h}$ and $\hat{S}^h_{t+h}$ from $\tilde{I}^h_{t}$ and $\tilde{S}^h_{t}$, that express the new number of susceptible and infectious individuals in each node before the displacement of individuals.
    \item Then, the new quantities of susceptible and infectious individuals are given by $\tilde{S}^h_{t+h}=\hat{S}^h_{t+h}{}^T Q$ and $\hat{S}^h_{t+h}{}^T Q$ respectively. This operation expresses the diffusion of the individuals in the graph according to the matrix $Q$.
\end{itemize}





\noindent We obtain by the same reasoning as before


 \begin{equation}
 \label{forecast2_}
 \displaystyle \mathbb{E}\left(\tilde{I}^h_{t+h}|\tilde{I}^h_t, \tilde{S}^h_t\right)=Q^T\left(\tilde{I}^h_t+\tilde{S}^h_t\left(1-e^{-\lambda {\tilde{I}^h_t} h}\right)-\tilde{I}^h_t\left(1-e^{-\gamma  h}\right)\right),
 \end{equation}
where 
$$
\displaystyle\lambda:=(\lambda_1,...,\lambda_n)^T\in(\mathbb{R}^+)^n,\ \gamma:=(\gamma_1,...,\gamma_n)^T\in(\mathbb{R}^+)^n,
$$ 
and $\tilde{I}^h_t \tilde{S}^h_t$ denotes the vector 
$$
(S^h_1(t)I^h_1(t),S^h_2(t)I^h_2(t),...,S^h_n(t)I^h_n(t))^T,
$$ 
and so on. For a function $f:\mathbb{R}\rightarrow \mathbb{R}$, the notation 
$$
\displaystyle f(x_1,\cdots,x_n):=\left(f(x_1),f(x_2),...,f(x_n)\right)^T
$$ 
was used in \eqref{forecast2_} for the exponential function, and it will also be used in the sequel. 


\paragraph{Convergence of the process.} Let us just indicate 
how the reasoning of Section \ref{sec:onenode} can be adapted in our context to get the same type of results. Denote by
$$\displaystyle N'_{\lambda, \gamma, h}(x_1,\cdots,x_n):=\begin{pmatrix} N_{\lambda_1, \gamma_1,h}(x_1)\\ \vdots\\
 N_{\lambda_n,\gamma_n, h}(x_n) \end{pmatrix},$$
and by 
\begin{equation}\label{forecastnd}\displaystyle \text{forecast}_{\lambda, \gamma, h}(u_1,v_1,\cdots,u_n,v_n):= N'_{\lambda, \gamma, h}(u_1,\cdots,u_n)\begin{pmatrix}
 u_1 & \ldots & u_n \\ v_1 & \ldots & v_n
 \end{pmatrix}Q^h.
\end{equation}
\noindent This last formula is reduced to \eqref{def_forecast} when there is a single node. A proof similar to the one of Lemma \ref{lem_martingale} shows that the following sequence is a martingale: 


$$\displaystyle \left(\begin{pmatrix}
 \prod_{i=1}^{k-1}N^{-1}_{\lambda_1, \gamma_1,h}\left(I^h_1\right)\begin{pmatrix}
I^h_1 \\ S^h_1
\end{pmatrix}, & \cdots &,
 \prod_{i=1}^{k-1}N^{-1}_{\lambda_n, \gamma_n,h}\left(I^h_n\right)\begin{pmatrix}
I^h_n\\ S^h_n
\end{pmatrix}
\end{pmatrix} Q^{-k+1}\right)_{k\in\mathbb{N}^*}.$$

\noindent Thus the statements of Theorem \ref{martingale} and Theorem \ref{MeanField} and their proofs carry over to our present situation. We conclude as in Subsection \ref{sub:stud} that {\it{if each coordinate of the vector $\lambda$ satisfies Conditions \eqref{growth} and \eqref{Condition2}, then  for all $t\in \mathbb{N}h$, the variables $\frac{\tilde{S}_t}{N}$ and  $\frac{\tilde{I}_t}{N}$ converge in probability towards a constant when $N\rightarrow +\infty$.}}

\subsection{Differential equations and forecasting. }
In this subsection we shall derive a differential equation expressing the evolution of the expectations of $\tilde{S}_t$ and $\tilde{I}_t$. According to \eqref{semigroup} we can suppose that there is a semigroup of stochastic matrices $t\mapsto Q^t$, with $Q^0 = I$, $Q^h=Q$, and $Q^t=I-\epsilon t\Delta^T+o(t)$. Here $\Delta$ is the Laplacian matrix defined in \eqref{def_laplacian}, and $\epsilon\geq 0$ is a parameter that expresses the scale between the time of displacement and the time of the contamination process. Denote by $\displaystyle \tilde{i}(t):= \lim_{h\rightarrow 0} \mathbb{E}(\tilde{I}^h_t)$ and by $\displaystyle \tilde{s}(t):= \lim_{h\rightarrow 0} \mathbb{E}(\tilde{S}^h_t)$. These quantities are well defined thanks because of Corollary \ref{limit2} below. Indeed, thanks to the commutativity of the operations of taking expectation of a vector of random variables and multiplying by the matrix $Q^h$, the steps leading to Corollary \ref{limit} in Appendix \ref{appendiceB} are still valid. It can be restated as:

\begin{coro}\label{limit2}
\noindent $\forall t\in \mathbb{N}h$, the limits $\displaystyle\lim_{l\rightarrow 0^+}\mathbb{E}(\tilde{I}^l_t)$ and $\displaystyle\lim_{l\rightarrow 0^+}\mathbb{E}(\tilde{S}^l_t)$ exist, and the convergence is uniform.\\
Moreover $\displaystyle\lim_{l\rightarrow 0^+}\mathbb{E}(\tilde{I}^l_t)=\mathbb{E}(\tilde{I}_t^h) + \tilde{O}(h^2)$ and $\displaystyle\lim_{l\rightarrow 0^+}\mathbb{E}(\tilde{S}^l_t)=\mathbb{E}(\tilde{S}_t^h) + \tilde{O}(h^2)$, where $\tilde{O}(\cdot)$ denote a vector of functions $O(\cdot)$ that does not depend on $t$. 
\end{coro}

\noindent By the same reasoning we can adapt Proposition \ref{justify_forecast_3} in Appendix B to our new setting. In particular, for some function $O(\cdot)$ which is independent of $k$, we obtain the following equation that we will use in Subsection \ref{sub:forecast} for forecasting purposes.

\begin{equation}\label{forecast2D}\displaystyle \begin{pmatrix}\mathbb{E}\left({S}^1_{kh}\right) \\ \mathbb{E}\left({I}^1_{kh}\right) \\ \vdots \\ \mathbb{E}\left({S}^n_{kh}\right) \\ \mathbb{E}\left({I}^n_{kh}\right) \end{pmatrix}=\text{forecast}_{\lambda, h}^{(k-1)}(\tilde{S}_0, \tilde{I}_0).\end{equation}
 
\noindent By letting $h\rightarrow 0$ in \eqref{forecast2_} after taking expectation, as previously in subsection \ref{sub:exp}, we obtain from Corollary \ref{limit2} the following result:

\begin{pro}
The vector valued functions $t\mapsto \tilde{i}(t)$ and $t\mapsto \tilde{s}(t)$ are differentiable and satisfy the following differential equations, that reduce to \eqref{sys1} when the graph underlying the model has a single node:

\begin{equation}\label{with_laplacian}
\left\lbrace
\begin{array}{rcl}
\displaystyle
\frac{\partial\tilde{s}(t)}{\partial t} &=&\displaystyle -{\lambda}\tilde{s}(t)\tilde{i}(t)-\epsilon \Delta^T \tilde{s}(t) \\
& & \\
\displaystyle\frac{\partial \tilde{i}(t)}{\partial t} &=&\displaystyle {\lambda}\tilde{s}(t)\tilde{i}(t)-\gamma \tilde{i}(t)-\epsilon \Delta^T \tilde{i}(t),
\end{array}
\right.
\end{equation}
where $\epsilon\geq 0$ is a parameter that expresses the scale between the time of displacement and the time of the contamination process.
\end{pro}

\begin{rmq}
By summing up the coordinates in \eqref{with_laplacian}, we obtain an equation that explains the slowing down of infection due to diffusion, compared to the classical SIR model with a single compartment, such as expressed in \eqref{sys1}. Indeed, denoting by $i(t)$ and $s(t)$ the total number of infected and susceptible individuals at time $t$ respectively, we obtain:

\begin{eqnarray}\displaystyle
\frac{\partial {s}(t)}{\partial t} &=& -n\left\langle {\lambda } \tilde{s}(t) \mid \tilde{i}(t)\right\rangle \\
\frac{\partial {i}(t)}{\partial t} &=& n\left\langle {\lambda } \tilde{s}(t)\mid \tilde{i}(t)\right\rangle - n\gamma i(t),
\end{eqnarray}

\noindent where $\displaystyle \langle \cdot \mid \cdot\rangle$ denotes the scalar product on $R^n$. The Cauchy-Schwarz inequality implies that the instantaneous speed of the propagation is maximal when the vector of the number of infectious individuals in each node is proportional to that of the number of susceptible individuals, which actually does not happen in reality, since the onset of the epidemic is spatially localized. This fact can be interpreted by saying that when compared to a standard SIR model, the diffusion effects slow down the contagion.
The greater the value of the diffusion parameter $\epsilon$ is, the sooner the diffusion tends to mix individuals and evolves the distribution of individuals closer to homogeneity, and thus approaches the picture given by a standard SIR model. We will see in Section \ref{sec:num} that such diffusion effects can also explain the occurrence of multiple front waves.
\end{rmq}

Note also that, thanks to \eqref{MCMCdiffusion}, for $\epsilon>0$ sufficiently small, the solutions of the equations \eqref{with_laplacian} are well approximated by those of the following equation:

\begin{equation}\label{without_laplacian}
\left\lbrace
\begin{array}{r c l}\displaystyle

\displaystyle\frac{\partial \tilde{s}(t)}{\partial t} &=&\displaystyle P_{\epsilon}\left( \lambda \tilde{s}(t) \tilde{i}(t)\right)\\
& & \\
\displaystyle\frac{\partial  \tilde{i}(t)}{\partial t} &=&\displaystyle P_{\epsilon}\left( {\lambda} \tilde{s}(t) \tilde{i}(t)-\gamma  \tilde{i}(t)\right),
\end{array}
\right.
\end{equation}

where $P_{\epsilon}$ is the matrix obtained by the MCMC method from an adjacency matrix of a graph whose diagonal values of are $\left \lfloor{\frac{1}{\epsilon}}\right \rfloor$. We recover equations (7) of \cite{BC}, together with a statistical interpretation, that allows us to use them in both contexts, i.e. in big or small populations.

\subsection{Illustration.}

Let us illustrate equations \eqref{with_laplacian} with the case where there are just two rooms, each containing initially 1500 individuals. Two individuals are initially infected the room 1. Here 
$$
\displaystyle P=
\begin{pmatrix}
9.99999992e-01 & 8.00000000e-09\\
 8.00000000e-09 & 9.99999992e-01\\ 
 \end{pmatrix}.
$$
Figures \ref{fig:2_pick_with_laplacian} and \ref{fig:2_pick_with_laplacian2} illustrate Equation \eqref{with_laplacian} with two nodes, each with $N=110000$ individuals, $\lambda = (0.3,0.3)$, $\gamma=(0.02,0.02)$, close to the coefficients found by \cite{BRBP} for the Covid-19 pandemic when the unit of time is one day.
The two peaks correspond to a wave for each room, since here the diffusion speed was chosen to be very slow.

\begin{figure}[!h]
     \centering
     \begin{subfigure}[b]{0.49\textwidth}
         \centering
    \includegraphics[width=\textwidth]{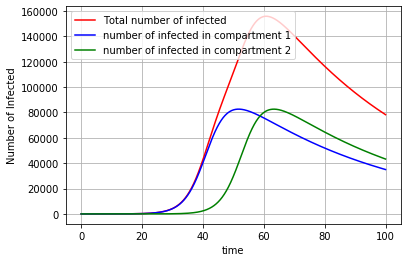}
    \caption{Dynamics for 2 rooms. $\epsilon=0.1$.}
    \label{fig:2_pick_with_laplacian}
     \end{subfigure}
     \hfill
     \begin{subfigure}[b]{0.49\textwidth}
         \centering
    \includegraphics[width=\textwidth]{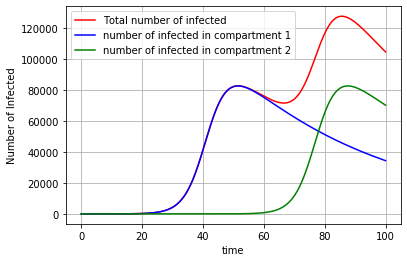}
    \caption{Dynamics for 2 rooms. $\epsilon=0.0001$.}
    \label{fig:2_pick_with_laplacian2}
     \end{subfigure}
\end{figure}

\begin{rmq}
Let us make a few remarks on possible generalizations. The model of diffusion on a non-oriented graph presented here could be generalized without difficulty to directed graphs, if the need for modeling there arose. Indeed the loss of symmetry in the $P$ matrix does not have an impact on the calculations. The addition of states other than Susceptible, Infected and Recovered, for example by structuring the population by age, is also possible. It seems also meaningful to replace the Markov diffusion process by a "wilder" process such as a Levy Flight in the sense of \cite{Mic}, by modifying \eqref{MCMCdiffusion} using equations (4.4) and (4.5) in this book. Indeed, this stochastic process, under certain circumstances, models well the human walk \cite{RHSL}. 
\end{rmq}

\subsection{The case of several groups.}\label{several_group}
In the context of a small population, we can  consider several groups with a schedule for each, i.e. for each group given probabilities for each individual of being in each node, as illustrated in Subsection \ref{sub:intro}. Specifically, we investigate now $K$ groups instead of one, with transition matrices $P^1,\ldots,P^K$. The same reasoning as in Subsection \ref{sub:exp} gives
$$\displaystyle \mathbb{E}\left(\tilde{I}^k_{t+1}|\tilde{I}_t,\tilde{S}_t\right) = \left(\tilde{S}^k_t\left(1-e^{-\lambda h \sum_{i=1}^K{\tilde{I}^i_t} }\right)+e^{-\gamma h}\tilde{I}_t\right)^T P^k.$$We have denoted by $\tilde{S}^k_{t}$, $\tilde{I}^k_{t}$ the vectors whose $N$ components are respectively the number of susceptible and infectious individuals of group $k$ for each rooms. (We have not indicated the dependence on the time step $h$ for simplicity). Denote by $\displaystyle \tilde{i}^k(t):= \lim_{h\rightarrow 0} \mathbb{E}(\tilde{I}^k_t)$ and by $\displaystyle \tilde{s}^k(t):= \lim_{h\rightarrow 0} \mathbb{E}(\tilde{I}^k_t)$ for $k=1,\cdots,K$. These quantities are well defined thanks to Corollary \ref{limit2}. In continuous time we obtain as before corresponding equations:

\begin{equation}\label{OD2D}
\displaystyle
\left\lbrace
\begin{array}{r c l}
\displaystyle\frac{\partial \tilde{s}^k(t)}{\partial t} &=&\displaystyle -\lambda\tilde{s}^k(t) \sum_{r=1}^K{\tilde{i}^r(t)}-\epsilon \Delta^T \tilde{s}^k(t) \\
& &\\
\displaystyle\frac{\partial\tilde{i}^k(t)}{\partial t} &=&\displaystyle \lambda\tilde{s}^k(t) \sum_{r=1}^K{\tilde{i}^r(t)}-\gamma\tilde{i}^k(t)-\epsilon \Delta^T \tilde{i}^k(t)
\end{array}
\right.
\end{equation}
\section{Numerical experiments.}\label{sec:num}

\subsection{Solutions for the ODE (\ref{with_laplacian}) 
of the model}

{\bf{a) SIR model. }} We simulate the spread of an epidemic inside a population moving in a structure with $2$ rooms. Each room contains initially 150 individuals. Then $2$ individuals are contaminated in room $1$. At each step, each individual moves to the other room with a probability of $0.08$, or stays. Then, the epidemic spreads in each room according to \eqref{infection_}. The simulation is represented by the blue curve of Figure \ref{fig:forecasting}.

\begin{figure}[!h]
     \centering
     \begin{subfigure}[b]{0.49\textwidth}
    \centering
    \includegraphics[width=\textwidth]{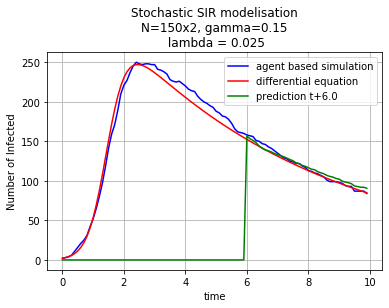}
    \caption{Evolution of an epidemic, 2 nodes.}
    \label{fig:forecasting}
     \end{subfigure}
    \centering
     \begin{subfigure}[b]{0.49\textwidth}
    \centering
    \includegraphics[width=\textwidth]{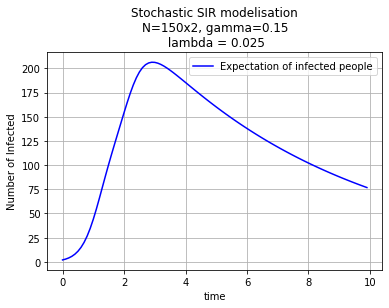}
    \caption{Differential equation \eqref{with_laplacian}.}
    \label{fig:with_laplacian}
        \end{subfigure}
        \caption{Dynamics of the expectation of the number of infected individuals for 2 nodes: simulation, differential equations, forecasting.}
\end{figure}

The green curve in the figure \ref{fig:forecasting} depicts the predictions \eqref{forecastnd} and \eqref{forecast2D}, with $\displaystyle Q=\begin{pmatrix} 0.92 & 0.08\\
 0.08 & 0.92\\ \end{pmatrix}$ here used in \eqref{forecastnd}. The formula \eqref{forecast2D} is applied iteratively until the desired date of prediction, at the jump of the green curve. \\
 Red curve is given by the differential equation \eqref{without_laplacian}. Compare the result with the differential equation given by \eqref{with_laplacian}, that is representated in Figure \ref{fig:with_laplacian}.

Figure \ref{fig:forecasting_3} presents such forecasts with other parameters, with 2 rooms each containing 800 individuals, and under the same procedure for initial contamination.\\

\begin{figure}[!h]
     \centering
     \begin{subfigure}[b]{0.49\textwidth}
    \centering
    \includegraphics[width=\textwidth]{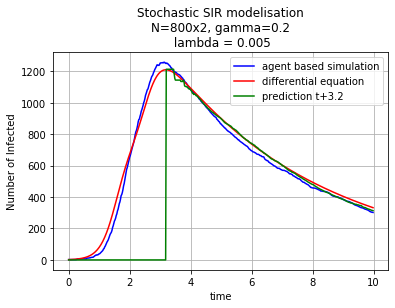}
    \caption{SIR model with 2 nodes and forecasting.}
    \label{fig:forecasting_3}
     \end{subfigure}
    \centering
     \begin{subfigure}[b]{0.49\textwidth}
    \centering
    \includegraphics[width=\textwidth]{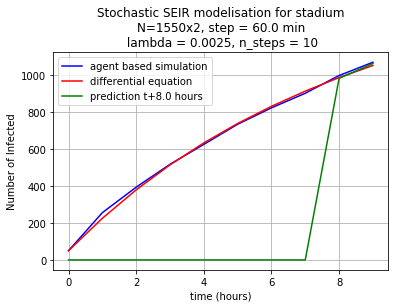}
    \caption{SEI model with 2 nodes and forecasting. }
    \label{fig:forecast}
        \end{subfigure}
 \caption{Dynamics of the expectation of the number of infected individuals for 2 nodes: other parameters.}
\end{figure}

{\bf{b) SEI model. }} 
Here E is the compartment containing the exposed individuals. It is the subset of those individuals who are suffering from
the illness, but since they are still in the incubation phase, they are not yet contagious. Let us compare the simulation with the simulation of a model slightly different from \eqref{infection_}, more suited to the study of the evolution of an epidemic in short time, on a one-day scale. Inside a stadium, for example, the contagion may take place within a period of 9 hours. We simulate the spread of an epidemic into a population moving in a structure with $2$ rooms. Each room contains initially 1550 individuals. Then $50$ individuals are contaminated in room $1$. At each step, each individual moves to the other room with a probability of $0.008$, or stays. Then, the epidemic spreads in each room according to the following SEI dynamics, where we use the same notations that \eqref{infection_}:

\begin{equation} 
\label{infection_stadium}
\displaystyle E^h_{t+h}=E^h_t+\sum_{i=1}^{S^h_t}X_i.
\end{equation} 

Indeed, it is here assumed that, during one day, newly infected individuals are not yet contagious, being in an Exposed state, denoted by $E$. We suppose also that there is no recovering during a day. We have used $\lambda=(0.0025,0.0025).$ The simulation is shown by the blue curve of Figure \ref{fig:forecast}. Ten steps are used for drawing this curve, but the number of steps is essentially inessential.

The green curve in the same figure  uses for prediction the following formula, iteratively applied, derived as \eqref{forecast2D}, and using the same notation:

\begin{equation}\label{forecast2__}\displaystyle \mathbb{E}(\tilde{I}_{t+h}|\tilde{I}_t, \tilde{S}_t, \tilde{E}_t)=P^T\left(\tilde{E}_t+\tilde{S_t}(1-e^{-\lambda {\tilde{I}_t} h})\right).\end{equation}

\noindent Here $\displaystyle P=\begin{pmatrix} 0.992 & 0.008\\
 0.008 & 0.992\\ \end{pmatrix}$, and $\displaystyle \Tilde{I}=\begin{pmatrix} I_0\\
 I_1 \end{pmatrix}$, where $I_i$ indicates the number of infected in room $i$. The notation is the same for the vectors $\Tilde{S}$, $\Tilde{E}$, $\Tilde{R}$ of susceptible, exposed and recovered individuals, respectively. This formula is applied iteratively until the desired date of prediction, here 8h later (the jump before 8h is an artefact of visualisation and obviously does not really exist).\\
 
\noindent Red curve is the solution of the following differential equation, comparable to \eqref{without_laplacian}:
 
 \begin{equation}
 \left\lbrace
 \begin{array}{r c l}
\displaystyle\frac{\partial \tilde{S}}{\partial t} &=&\displaystyle P_{\epsilon}\left( -\lambda \mathbb{E}(\tilde{S}_t{ )\mathbb{E}(\tilde{I}_t})\right) \\
& & \\
\displaystyle\frac{\partial \tilde{E}}{\partial t} &=&\displaystyle P_{\epsilon}\left( \lambda \mathbb{E}(\tilde{S}_t)\mathbb{E}( \tilde{I}_t)\right).
\end{array}
\right.
\end{equation}

\subsection{Forecasting}\label{sub:forecast} Now let us illustrate how \eqref{with_laplacian} can be used for forecasting. We pursue the idea of \cite{Torus}, according to which the evolution of the epidemic can be decomposed into several fronts waves. This approach makes it possible to detect the rise of a wave very early and to estimate its strength. First we collect the number of daily cases in Singapore (\cite{Data}) for the Covid 19 pandemic. Neglecting the number of deaths, we estimate the number of active cases by:
\begin{equation}
\begin{array}{rcl}
\text{Active Cases at day t} &=& \text{Cumulative Cases until day t}\\
 & &\\
     &-&\text{Cumulative Recovered until day t.}
\end{array}
\end{equation}

\noindent We have also supposed that only a percentage of the daily cases has been really detected. We make this assumption in order to reproduce one more time the phenomenon of asymptomatic cases present in the Covid-19 epidemic. An error function has been implemented, as the mean square of the difference between active cases at Singapore, multiplied by the rate of detection previously mentioned, and solutions of \eqref{forecast_} or \eqref{forecast2_}. The parameters in these equations minimising this error function, including the detection rate, have been obtained by the basin-hopping method. A total population of $N=50000000$ has been used, distributed between the $2$ subpopulations according to an unknown proportion that is also a parameter of the error function. Figure \ref{fig:fitting.png} shows the active cases compared to the model. Data until $t=180$, to the left to the dashed green vertical line, was used for training. This method can be an alternative to \cite{Torus} to detect the onset of a wave.

\begin{figure}[!h]
    \centering
    \includegraphics[width=10cm]{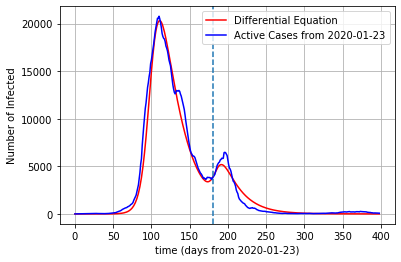}
    \caption{Active Cases for Covid-19 pandemic in Singapore: Data and Model.}
    \label{fig:fitting.png}
\end{figure}

\subsection{Scenarios for multiple waves}\label{sub:scen}

Let us first illustrate the effects of multiples nodes or lockdown of the generation of several epidemic waves. Limiting movement measures on a large scale will be modeled by a decrease of the diffusion coefficient $\epsilon\geq 0$ in equation \eqref{with_laplacian}, while social distancing and lockdown measures will be modeled by a decrease of the incidence coefficients $\lambda$ in the same equation. Figure \ref{fig:without_containment.png} shows left the expected number of infectious individuals given by \eqref{with_laplacian} with 1 node (in this case, $\epsilon=0$). Here the values of the parameters are $N = 100000$, $\lambda=0.02$, $\gamma=0.01$. The right graphics use the same parameters, except that $\lambda=0.012$ until $t=10000$. At this time, it remains $68639$ susceptible individuals. Then $\lambda$ increases to $0.02$, simulating an easing of lockdown measures. Then a second wave arises.\\
Figure \ref{fig:without_containment2.png} shows left the expected number of infectious individuals given by \eqref{with_laplacian} with 4 nodes, linked in a line. The initial number of susceptible individuals in each of these nodes are $50000000000$, $50000$, $500000$, $50000000000$ respectively. Then $10000$ infectious individuals are introduced in node $1$. Here the values parameters are $\epsilon = 0.02$, $\lambda = (0.3, 0.25, 0.25, 0.25)$, $\gamma = (0.1, 0.1, 0.15, 0.1)$. Only the number of infectious individuals for the two central nodes are showed. On the right graphics, before $t=1100$, distancing or lockdown measures are simulated by dividing by $2$ the value of the vector $\lambda$. On the other hand, in the left graphics of Figure \ref{fig:lockdown.png}, measures limiting movements on a large scale are simulated before $t= 1100$ by dividing by $10$ the diffusion coefficient $\epsilon$. It can be seen that the height of the contamination peak will be reduced.\\
Figure \ref{fig:without_containment3.png} shows left the expected number of infectious individuals given by \eqref{with_laplacian} on the same graph with 4 nodes than previously. The initial number of susceptible individuals in each of these nodes are $500000000$, $50000$, $500000$, $500000000$ respectively. Then $10000$ infectious individuals are introduced in node $1$, and $1110$ in node $4$. Here the values of the parameters are $\epsilon = 0.1$, $\lambda = (0.3, 0.25, 0.25, 0.25)$, $\gamma = (0.1, 0.1, 0.15, 0.09)$. On the right graphics, after $t=800$, distancing measures is simulated by dividing by $2$ the vector $\lambda$. On the other hand, on the right graphics of Figure \ref{fig:lockdown.png}, measures limiting movements on a large scale are simulated before $t= 800$ by dividing the diffusion coefficient $\epsilon$ by $10$. It can again be seen that the height of the peak of contamination is reduced.\\
We conclude that multiple front waves can occur by diffusion effects, between sub-populations or geographical zones according to the signification of the nodes. These waves can again be multiplied by the effects of distancing or lockdown measures, leading to more complex effects.

\begin{figure}[h!]
    \centering
    \includegraphics[width=6cm]{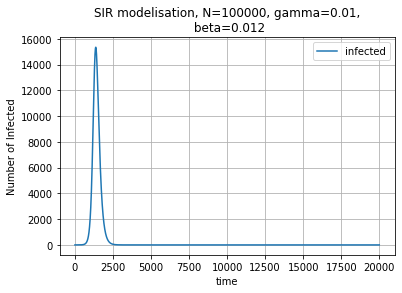}
    \includegraphics[width=6cm]{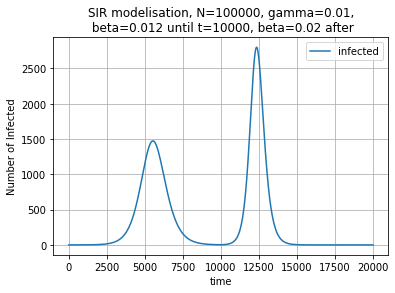}
    \caption{Dynamics of an epidemic without lockdown measures (left) and with distancing measures until $t=10000$ (right), 1 node.}
    \label{fig:without_containment.png}
\end{figure}

\begin{figure}[!h]
    \centering
    \includegraphics[width=6cm]{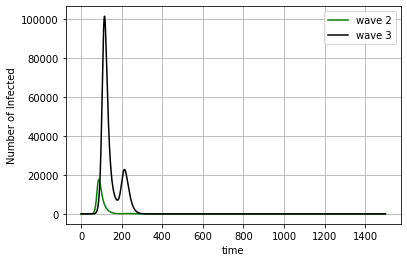}
    \includegraphics[width=6cm]{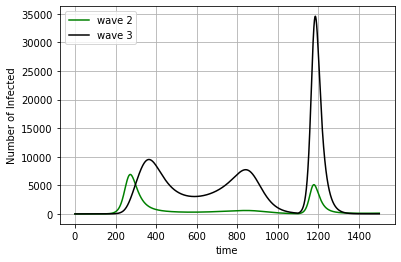}
    \caption{Dynamics of an epidemic without distancing measures (left) and with distancing measures until $t=1100$ (right), 4 nodes.}
    \label{fig:without_containment2.png}
\end{figure}

\begin{figure}[!h]
    \centering
    \includegraphics[width=6cm]{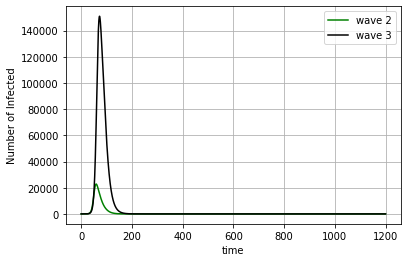}
    \includegraphics[width=6cm]{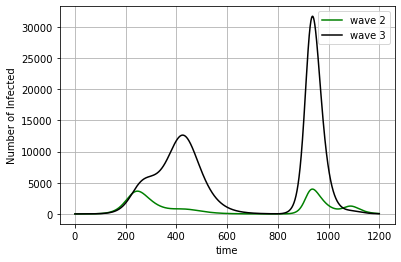}
    \caption{Dynamics of an epidemic without distancing measures (left) and with distancing measures until $t=800$ (right), 4 nodes.}
    \label{fig:without_containment3.png}
\end{figure}

\begin{figure}[!h]
    \centering
    \includegraphics[width=6cm]{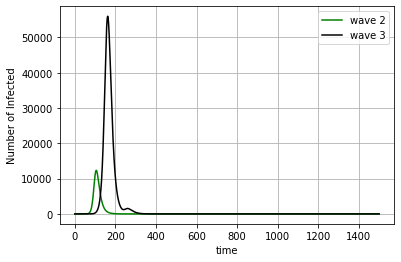}
    \includegraphics[width=6cm]{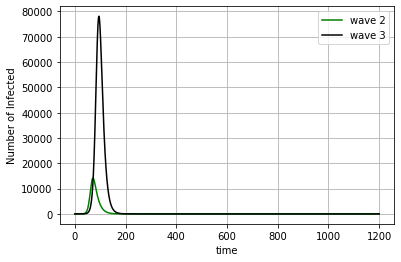}
    \caption{Dynamics of an epidemic with limiting movement measures on a large scale, with different parameters on the left and right.}
    \label{fig:lockdown.png}
\end{figure}

\newpage
\begin{appendices}\label{appendices}

\section{General lemmas}\label{appendixA}

Let us gather here some general lemmas that we have used so far. The first one results from a straightforward computation.

\begin{lem}
Let $\mathcal{B}\subset\mathcal{B}'$ two $\sigma$-algebras defined on a space $A$, and $X$ a random with finite variance defined on $A$. Then \begin{equation}\label{cond-variance}
V^{\mathcal{B}}(X) = V^{\mathcal{B}}(\mathbb{E}^{\mathcal{B}'}(X))+\mathbb{E}^{\mathcal{B}}(V^{\mathcal{B}'}(X)).
\end{equation}
\end{lem}

\begin{lem}\label{lemma}
Let $\mathcal{B}$ a $\sigma$-algebra defined on a space $A$, and $X\geq 0$ a random variable with finite variance defined on $A$. Then the function $$\displaystyle \mathcal{B}'\mapsto V^{\mathcal{B}}\mathbb{E}^{\mathcal{B}'}(X)$$ is increasing on the set of $\sigma$-algebras containing $\mathcal{B}$, and bounded from above by $V^{\mathcal{B}}(X)$. Moreover,
$$\displaystyle \mathcal{B}'\mapsto \mathbb{E}^{\mathcal{B}}V^{\mathcal{B}'}(X)$$ is decreasing on the set of $\sigma$-algebras containing $\mathcal{B}$, and bounded from above by $V^{\mathcal{B}}(X)$.
\end{lem}

\begin{proof}
Let three $\sigma$-algebras $\mathcal{B}\subset \mathcal{B}'\subset \mathcal{B}''$ be given. Equation \eqref{cond-variance} implies $$\displaystyle V^{\mathcal{B}}\mathbb{E}^{\mathcal{B}'}(X)= V^{\mathcal{B}}\mathbb{E}^{\mathcal{B}'}\mathbb{E}^{\mathcal{B}''}(X)\leq V^{\mathcal{B}}\mathbb{E}^{\mathcal{B}''}(X)\leq V^{\mathcal{B}}(X).$$
This proves the first statement. The second then follows from \eqref{cond-variance}, that gives $$\mathbb{E}^{\mathcal{B}}(V^{\mathcal{B}'}(X))=V^{\mathcal{B}}(X) - V^{\mathcal{B}}(\mathbb{E}^{\mathcal{B}'}(X)).$$
\end{proof}






\section{Technical lemmas}\label{appendiceB}

This subsection gives a justification for using \eqref{forecast2_} for long term forecasting. We refer to Subsection \ref{sub:stud} for the notations.

Recall that the function $\text{forecast}_{\lambda,h}:\mathbb{R}^2\rightarrow \mathbb{R}^2$ is defined by \eqref{def_forecast}. 
For convenience, for all $n\in \mathbb{N}$, we denote by $\mathcal{T}^h_n$ the $\sigma$-algebra generated by $I_{nh}$ and $S_{nh}$. Recall that, being a random variable $X$, the conditional expectation $\displaystyle \mathbb{E}\left(X\mid \mathcal{T}^h_k\right)$ is a random variable that is a unique measurable function, a.e. defined, of the pairing $(S^h_{kh},I^h_{kh})$ (see e.g. Proposition 3 in \cite{R}). Let us denote a continuous representing of this function, if exists, by $\displaystyle \mathbb{E}\left(X\mid \mathcal{T}^h_k\right)\left(\cdot ,\cdot\right)$. (In literature, the notation  $\displaystyle (x,y)\mapsto \mathbb{E}\left(X\mid S^h_{kh}=x,I^h_{kh}=y\right)$ is also used). Then \eqref{matrix_forecast} implies, for all $t=kh$, $k\in\mathbb{N}h$ and $x,y\in\mathbb{R}^+$,

\begin{equation}
\label{matrix_forecast2}
\displaystyle \text{forecast}_{\lambda,h}(x, y) =
\begin{pmatrix}
\mathbb{E}\left(I^h_{t+h}\mid\mathcal{T}^h_k\right)(x, y) \\ 
\mathbb{E}\left(S^h_{t+h}\mid\mathcal{T}^h_k \right)(x, y) 
\end{pmatrix}=N_{\lambda, h}(x)
\begin{pmatrix} x \\ 
y
\end{pmatrix}.
\end{equation}
We will use the following result:

\begin{pro}\label{justify_cond_exp}
For all $h>0$ and $t\in \mathbb{N}h$ we have $\displaystyle Cov^{S^h_t, I^h_t}\left(S^h_{t+h}, I^h_{t+h}\right)=O(h^2)$, where the function $O(\cdot)$ is independent of $t$. More precisely, $\displaystyle O\left(h^2\right)\leq \lambda^2N^3h^2$, where $N$ is the total population.
\end{pro}

\begin{proof}
Using the same notations as in \eqref{infection_}, we can compute:
$$\displaystyle \mathbb{E}(I^h_{t+h}S^h_{t+h}\mid S^h_t, I^h_t)=\mathbb{E}\left( \left( I^h_t+\sum_{i=1}^{S_t}X_i-\sum_{i=1}^{I_t}Y_i  \right)\left(S^h_t-\sum_{i=1}^{S^h_t}X_i\right) \mid S^h_t, I^h_t\right)$$
$$\displaystyle = I^h_t S^h_t + S^h_t \mathbb{E}\left(  \sum_{i=1}^{S_t}X_i \mid S^h_t, I^h_t \right) - S^h_t \mathbb{E}\left( -\sum_{i=1}^{I_t}Y_i \mid S^h_t, I^h_t \right) -I^h_t \mathbb{E}\left( \sum_{i=1}^{I_t}Y_i \mid S^h_t, I^h_t \right)$$ $$\displaystyle + \mathbb{E}\left( \left(\sum_{i=1}^{I^h_t}Y_i\right)\left(  \sum_{i=1}^{S^h_t}X_i \right) \mid S^h_t, I^h_t \right)+ \mathbb{E}\left(\left(\sum_{i=1}^{S^h_t}X_i \right)^2 \mid S^h_t, I^h_t\right).$$
Recalling the independence of the Bernoulli's variables $X_i$ and $Y_i$ according to the $\sigma$-algebra generated by both $I^h_t$ and $S^h_t$, we see that the last term is equal to
$$\displaystyle I^h_t S^h_t + S^h_t \mathbb{E}\left(  \sum_{i=1}^{S^h_t}X_i \mid S^h_t, I^h_t \right) - S^h_t \mathbb{E}\left( \sum_{i=1}^{I^h_t}Y_i\mid S^h_t, I^h_t \right) -I^h_t \mathbb{E}\left( \sum_{i=1}^{I^h_t}Y_i \mid S^h_t, I^h_t \right)$$ $$\displaystyle + \mathbb{E}\left(\sum_{i=1}^{I^h_t}Y_i \mid S^h_t, I^h_t\right) \mathbb{E}\left( \sum_{i=1}^{S^h_t}X_i \mid S^h_t, I^h_t \right) + \mathbb{E}\left(  \sum_{i=1}^{S^h_t}X_i \mid S^h_t, I^h_t \right)^2 +V\left(  \sum_{i=1}^{S^h_t}X_i \mid S^h_t, I^h_t \right) $$
$$=\displaystyle \mathbb{E}\left( \left(I^h_t+S^h_t e^{-\lambda I^h_t h}- I^h_t e^{-\gamma h}\right)\left( S^h_t - S^h_t e^{-\lambda I^h_t h}\right) \mid S^h_t, I^h_t\right) +\lambda^2S^h_t (I^h_t)^2h^2$$
$$\displaystyle =  \mathbb{E}\left(I^h_{t+h}\mid S^h_t, I^h_t\right) \mathbb{E}\left(S^h_{t+h}\mid S^h_t, I^h_t\right)+\lambda^2 S^h_t (I^h_t)^2 h^2.$$
The conclusion follows from the boundedness of $S^h_t$ and $I^h_t$ by $N$ from above.

\end{proof}

\noindent One consequence is the following corollary. It shows that under condition \ref{growth} the number of susceptible individuals and infectious individuals become less correlated when $N$ is large and $h$ is small.

\begin{coro}\label{leqn3}
For all $h>0$ and $t\in \mathbb{N}h$ we have $$\displaystyle Cov(S^h_{t+h}, I^h_{t+h})\leq N^3\lambda^2(3t+h).$$
\end{coro}

\begin{proof}
The law of total covariance together with Proposition \ref{justify_cond_exp} give, for $h>0$ sufficiently small, $$\displaystyle Cov(S^h_{t+h}, I^h_{t+h})\leq Cov(S^h_{t}, I^h_{t})(1-\gamma h) + N^3\lambda^2(3h+h^2).$$ The conclusion follows by a recurrence argument.
\end{proof}

Let us denote a partial order relation on $\mathbb{R}^2$ as $\displaystyle\begin{pmatrix}u\\v\end{pmatrix}\leq \begin{pmatrix}u'\\v'\end{pmatrix}$ whenever $u\leq u'$ and $v\leq v'$. Denote also by $\displaystyle \text{forecast}^{(k)}_{\lambda,h}(x, y)$ the $k$-th iterate of the function $\text{forecast}_{\lambda,h}(x, y)$. We have:

\begin{lem}\label{uniform}
For $h>0$, there exist a constant $C_h>0$ such that for all $x,y\in \mathbb{R}^+$ and $k\in\mathbb{N}^*$,
\begin{equation}\label{bound}\displaystyle  \text{forecast}^{(k)}_{\lambda,h}(x, y)\leq C_h \begin{pmatrix} y\\x \end{pmatrix}.\end{equation}
\noindent Moreover, if Condition \eqref{growth} is verified, for $h>0$, there exist a constant $C'_h>0$ independent of $N$ such that for all $x,y\in \mathbb{R}^+$, $y\leq N$,
\begin{equation}\label{bound}\displaystyle {C'_h}\begin{pmatrix} y\\x \end{pmatrix}\leq \text{forecast}_{\lambda ,h}(x, y).\end{equation}

\end{lem}

\begin{proof}We can suppose $y\geq 1$. We can rewrite \eqref{def_forecast} as

\begin{equation}\label{forecast_matrix}\displaystyle \text{forecast}_{\lambda, h}(x, y) =N_{\lambda,h}(y) \begin{pmatrix} y \\ x\end{pmatrix}=\begin{pmatrix}e^{-\gamma h} & 1-e^{- \lambda  y  h}\\
0 & e^{- \lambda  y h}\\ \end{pmatrix}\begin{pmatrix} y \\ x\end{pmatrix}.\end{equation}
\noindent Thanks to \eqref{forecast_matrix}, we have for all $h,x,y\in\mathbb{R}^{+}$ and $k\in\mathbb{N}$:

\begin{equation}\label{forecastk}\displaystyle \text{forecast}^{(k)}_{\lambda,h}(x, y)\leq \begin{pmatrix}e^{-\gamma h} & 1\\
0 & e^{- \lambda  h}\\ \end{pmatrix}^k\begin{pmatrix}y \\ x\end{pmatrix}.\end{equation}




\noindent Then the first inequality of the statement follows from the fact that the eigenvalues of the matrices $\displaystyle \begin{pmatrix}e^{-\gamma h} & 1\\
0 & e^{- \lambda  h}\\ \end{pmatrix}$ are $<1$. The second inequality in the statement follows from $$\displaystyle\begin{pmatrix} e^{-k\gamma}y\\ \displaystyle e^{-k\lambda y}x\end{pmatrix}\leq \text{forecast}^{(k)}_{\lambda,h}(x, y).$$
\end{proof}

\noindent Now we can state the following result:

\begin{pro}\label{justify_forecast_3}

For all $t=kh$, $k\in \mathbb{N}$, the expectations $\mathbb{E}\left(S^h_{t}\right)$ and $\mathbb{E}\left(I^h_{t}\right)$ can be estimated from the initial conditions $S_0$, $I_0$, thanks to the $k^{th}$ iterate $\displaystyle\text{forecast}_{\lambda,h}^{(k)}$ of the function $\displaystyle \text{forecast}_{\lambda,h}$, by

\begin{equation}\label{one_}\displaystyle \begin{pmatrix}\mathbb{E}\left(S^h_{kh}\right) \\ \mathbb{E}\left(I^h_{kh}\right) \end{pmatrix}=\text{forecast}_{\lambda, h}^{(k)}(S_0, I_0).\end{equation}
In other words, for all $t\in\mathbb{N}h$ and $r\in\mathbb{N},$ one has

\begin{equation}\label{two_}\displaystyle \begin{pmatrix}\mathbb{E}(S^h_{t+rh}) \\ \mathbb{E}(I^h_{t+rh}) \end{pmatrix}=\text{forecast}_{\lambda, h}^{(r)}\left(\mathbb{E}\left(S^h_{t}\right), \mathbb{E}\left(I^h_{t}\right)\right).\end{equation}
 More generally, for $t=kh$ and $t'=rh$, $r\leq k$, we have\\
\begin{equation}\label{three_}\displaystyle \begin{pmatrix}\mathbb{E}(S^h_{t}\mid \mathcal{T}^h_r) \\ \mathbb{E}(I^h_{t}\mid \mathcal{T}^h_{r}) \end{pmatrix}=\text{forecast}_{\lambda, h}^{(k-r)}\left(S^h_{t'}, I^h_{t'}\right).\end{equation}
Moreover for all $k\in\mathbb{N}^*$, the value $\displaystyle\mathbb{E}\left(\mathbb{E}(I_{kh}),\mathbb{E}(S_{kh})\right)$ has a unique inverse by the function $\displaystyle\text{forecast}_{\lambda,h}$, given by:

\begin{equation}\label{inverse}
\text{N}^{-1}_{\lambda, h}(\mathbb{E}(I_{kh})) \begin{pmatrix} \mathbb{E}\left(I_{kh}\right)\\ \mathbb{E}\left(S_{kh}\right)\end{pmatrix} =\begin{pmatrix}\mathbb{E}\left(S_{(k-1)h}\right) \\ \mathbb{E}\left(I_{(k-1)h}\right)\end{pmatrix}.
\end{equation}
\end{pro}

\begin{proof}

Let us prove \eqref{three_} for the second coordinate. For convenience of the reader, we will not indicate the indices $\lambda, h$, supposed to be fixed, of the function $\displaystyle \text{forecast}_{\lambda, h}$. The proof for the first coordinate proceeds exactly in the same way. First note that for $t,t_0,t_1\in\mathbb{N}h$ with $t\geq t_0\geq t_1,$ we have
\begin{equation}
    \label{markovian}\mathbb{E}\left(\mathbb{E}\left(I^h_t\mid \mathcal{T}_{t_0}\right)\mid \mathcal{T}_{t_1}\right)=
   \mathbb{E}\left(I^h_t\mid \mathcal{T}_{t_1}\right).
\end{equation}
Indeed, by construction, according to \eqref{infection_}, the sequence $\displaystyle  \left(I^h_{k h},S^h_{k h}\right)_{k\in\mathbb{N}}$ is a (homogeneous) Markov chain. So, with the previous notations,
$$\displaystyle \mathbb{E}\left(\mathbb{E}\left(I^h_t\mid I^h_{t_0}, S^h_{t_0}\right)\mid \mathcal{T}_{t_1}\right)=
  \mathbb{E}\left(\mathbb{E}\left(I^h_t\mid I^h_{t_0}, S^h_{t_0}, I^h_{t_1}, S^h_{t_1}\right)\mid \mathcal{T}_{t_1}\right)=  \mathbb{E}\left(I^h_t\mid \mathcal{T}_{t_1}\right).$$

\noindent Let $g_{\lambda,h}(x,y):=y\left(1-e^{-\lambda x h}\right)$. Let the recurrence hypothesis be
$$\displaystyle \mathcal{P}_n\iff\forall l\leq n,\text{ }\forall k\in\mathbb{N}^*, \text{ } P_1\text{forecast}^{(l)}(I^h_{k h})=\mathbb{E}\left(I^h_{(k+l) h}\mid \mathcal{T}_k\right),$$
and
$$\displaystyle \mathcal{Q}_n\iff \forall t\in\mathbb{N}^*h,\text{ }\forall l\in \llbracket 1, n-1\rrbracket,\text{ }\forall k\in \llbracket 0,n-l\rrbracket,\text{ }$$

$$\displaystyle \mathbb{E}\left(g\left(I_{t}^h, S_{t}^h\right)\mid\mathcal{T}_k\right)=\displaystyle \mathbb{E}\left( g\left(\mathbb{E}\left(I^h_t\mid \mathcal{T}_{n-l}\right),\mathbb{E}\left(S^h_t\mid \mathcal{T}_{n-l}\right)\right)\mid \mathcal{T}_k\right).$$

\noindent $\displaystyle \mathcal{P}_1$ is true by \eqref{forecast_N}. Let us show that $\displaystyle \mathcal{P}_{n}\implies \mathcal{Q}_{n}.$ Suppose $\mathcal{P}_{n}$. Recall that $P_0:\mathbb{R}^2\rightarrow \mathbb{R}$ and $P_1:\mathbb{R}^2\rightarrow \mathbb{R}$ denote the functions that returns the first and second coordinate respectively. $\forall l\in  \llbracket 1,n-1\rrbracket$, $\forall k\in \llbracket 0,n-l\rrbracket$,

\begin{equation}
\begin{array}{r c l}
\displaystyle \mathbb{E}\left(g\left(I_{t}^h, S_{t}^h\right)\mid\mathcal{T}_k\right)&=&\displaystyle \mathbb{E}\left(P_1\text{forecast}(I^h_t, S^h_t)-(1-\gamma h)I_t^h\mid \mathcal{T}_k\right)\\
& &\\
&\overset{\eqref{markovian}}{=}&\displaystyle \mathbb{E}\left( \mathbb{E}\left( I_{t+h}^h - (1-\gamma h) I_{t}^h\mid \mathcal{T}_k\right)\right)\\
& &\\
&\overset{\mathcal{P}_{n}}{=}&\displaystyle \mathbb{E}\left( P_1\text{forecast}^{(l+1)}\left(I^h_{t-l h}, S^h_{t-l h}\right)-(1-\gamma h) P_1\text{forecast}^{(l)}\left(I^h_{t-l h}, S^h_{t-l h}\right)\mid \mathcal{T}_k\right)\\
& &\\
&=&\displaystyle \mathbb{E}\left( \mathbb{E}\left( P_1\text{forecast}\left(I^h_{t-l h}, S^h_{t-l h}\right)-(1-\gamma h)  I^h_{t-l h}\mid \mathcal{T}_{n-l}\right)\circ\text{forecast}^{(l)}\left(I^h_{t-l h},S^h_{t-l h}\right)  \right.\\
& &\\

& &\left.\mid \mathcal{T}_k\right)\\
& &\\
&=&\displaystyle  \mathbb{E}\left(\mathbb{E}\left(g\left(S_{t-l h}^h, I_{t-l h}^h\right)\mid \mathcal{T}_{n-l}\right)\circ\text{forecast}^{(l)}\left(S^h_{t-l h}, I^h_{t-l h}\right) \mid \mathcal{T}_k\right)\\
& &\\
&=&\displaystyle \mathbb{E}\left(g\left(P_0\text{forecast}^{(l)}\left(S^h_{t-l h}, I^h_{t-l h}\right) , P_1\text{forecast}^{(l)}\left(S^h_{t-l h}, I^h_{t-l h}\right) \right)\mid \mathcal{T}_k\right)\\
& &\\
&\overset{\mathcal{P}_{n}}{=}&\displaystyle \mathbb{E}\left( g\left(\mathbb{E}\left(I^h_t\mid \mathcal{T}_{n-l}\right),\mathbb{E}\left(S^h_t\mid \mathcal{T}_{n-l}\right)\right)\mid \mathcal{T}_k\right).
\end{array}
\end{equation}

\noindent Now let us show that $\mathcal{Q}_n\implies \mathcal{P}_{n+1}$. Suppose $\mathcal{P}_n$. Then for $t=k h$,

$$\displaystyle
\begin{array}{r c l}
P_0\text{forecast}^{(n+1)}(I^h_t,S^h_t)&=&P_0\text{forecast}\left(\mathbb{E}\left(I^h_{t+n h}\mid \mathcal{T}_k\right)\right)\\
& &\\
&=&\displaystyle (1-\gamma h) \mathbb{E}\left(I^h_{t+n h}\mid \mathcal{T}_k\right)+g\left(\mathbb{E}\left(I^h_{t+n h}\mid \mathcal{T}_k\right),\mathbb{E}\left(S^h_{t+n h}\mid \mathcal{T}_k\right) \right)\\
& &\\
&\overset{\mathcal{Q}_{n}}{=}&\displaystyle (1-\gamma h) \mathbb{E}\left(I^h_{t+n h}\mid \mathcal{T}_k\right)+ \mathbb{E}\left(g\left(S_{t+n h}^h, I_{t+n h}^h\right)\mid\mathcal{T}_k\right)\\
& &\\
&=&\displaystyle \mathbb{E}\left(\mathbb{E}\left( I_{t+(n+1) h}\mid \mathcal{T}_{n}\right)\mid \mathcal{T}_k\right)\\
& &\\
&\overset{\eqref{markovian}}{=}& \mathbb{E}\left( I_{t+(n+1) h}\mid \mathcal{T}_k\right).
\end{array}
$$
In summary, we have shown by recurrence that $\mathcal{P}_n$ is true for all $n\in \mathbb{N}^*$. In particular, \begin{equation}\label{result}\displaystyle \forall k\in \mathbb{N},\text{ } P_0\text{forecast}^{(k)}(I^h_h,S^h_h)=\mathbb{E}\left(I^h_{(k+1)h}\mid I^h_h, S^h_h\right).\end{equation}

\noindent It remains to prove that $\mathcal{P}_n$ is still true for $k=0$, i.e. that

$$\displaystyle \forall n\in \mathbb{N}^*,\text{ }\forall n\leq k,\text{ } P_0\text{forecast}^{(n)}(I^h_0,S^h_0)=\mathbb{E}\left(I^h_{n h}\mid \mathcal{T}_0\right)=\mathbb{E}\left(I^h_{n h}\right).$$

\noindent (The last equality follows from the fact that $I_0$ is constant, as initial condition, so $\mathcal{T}_0$ is trivial.) For this, it suffices to prove that the function $\displaystyle (t,t_1)\mapsto \mathbb{E}\left(I^h_t\mid \mathcal{T}_{\frac{t_1}{h}}\right)(x)$ is only depending on the difference $t-t_1$ and not on the actual value of $t$ and $t_1$. Then it follow, for $t=n h$, $n\in\mathbb{N}$,
\begin{equation}\label{translate}
    \displaystyle P_0\text{forecast}^{(n)}(I_0,S_0)\overset{\eqref{result}}{=}\mathbb{E}\left(I^h_{t+h}\mid \mathcal{T}_{1}\right)(I_0)=\mathbb{E}\left(I^h_{t}\mid \mathcal{T}_{0}\right)(I_0,S_0)=\mathbb{E}\left(I^h_{t}\right).
\end{equation}
The fact that the function $\displaystyle x\mapsto \mathbb{E}\left(I^h_t\mid\mathcal{T}_{\frac{t_1}{h}}\right)(x)$ is depending only on $t-t_1$ follows from the fact that the sequence $\displaystyle \left(I^h_{k h},S^h_{k h}\right)_{k\in\mathbb{N}}$ is a homogeneous Markov chain, but let us prove it directly. Let $t=k h$, $t'=k' h$, with $k,k'\in\mathbb{N}$ such that $k>k'$. Then \eqref{infection_} implies $\forall x,y \in \mathbb{N},$
$$\displaystyle \mathbb{E}\left( I^h_t\mid \mathcal{T}_{k'}\right)(x,y)=\mathbb{E}\left(\mathbb{E}\left( I^h_t\mid \mathcal{T}_{k'+1}\right)\mid \mathcal{T}_{k'}\right)(x,y)=\mathbb{E}\left(\text{forecast}^{(k-k'+1)}\left( I^h_{t'+h}\right)\mid \mathcal{T}_{k'}\right)(x,y)$$
$$=\displaystyle \mathbb{E}\left(P_0\text{forecast}^{(k-k'+1)}\left( I^h_{t'}+\sum_{i=1}^{S^h_{t'}}X_i-\sum_{i=1}^{I^h_{t'}}Y_i,S^h_t-\sum_{i=1}^{S^h_{t'}}X_i\right)\mid \mathcal{T}_{k'}\right)(x,y)$$ $$\displaystyle =\sum_{u,v\in\mathbb{N}}P_0  \text{forecast}^{(k-k'+1)}(x,y)\mathbb{P}\left(x+\sum_{i=1}^{x}U_i(x)-\sum_{i=1}^{y}Y_i=u,y-\sum_{i=1}^{y}U_i(x)=v\right),$$
where the $Y_i$'s are independent Bernoulli variables with law $1-e^{-\gamma h}$, and the $U_i(x)$'s are independent Bernoulli variables with law $1-e^{-\lambda x h}$. It appears that $\displaystyle (x,y)\mapsto \mathbb{E}\left( I^h_t\mid \mathcal{T}_{k'}\right)(x,y)$ is dependent on $t-t'$ but not on the actual value of $t$ and $t_1$.\\

\noindent Finally, \eqref{inverse} can be deduced from \eqref{two_} by 

\begin{equation}
    \begin{array}{r c l}
 \displaystyle \begin{pmatrix}\mathbb{E}(I_{kh})\\ \mathbb{E}(S_{kh})\end{pmatrix}       &=&\displaystyle \text{N}^{-1}_{\lambda,h}\left(\mathbb{E}(I_{kh})\right)\text{forecast}_{\lambda,h}\left(\mathbb{E}(I_{kh}),\mathbb{E}(S_{kh})\right) \\
         & & \\
 \displaystyle        &=&\displaystyle N_{ \lambda,h}^{-1}(\mathbb{E}(I_{kh}))\begin{pmatrix}\mathbb{E}\left(S_{(k+1)h}\right)\\ \mathbb{E}\left(I_{(k+1)h}\right)\end{pmatrix} .\\   

    \end{array}
\end{equation}

\end{proof}

Proposition \ref{justify_forecast_3} can be used alternatively to \eqref{with_laplacian} for forecasting. The latter method can presumably be transposed to more complex models, with memory effects for example (i.e. non Markovian models), when obtaining a differential equation would seem difficult to obtain. Let us state the following corollary:

\begin{coro}\label{product_exp} Let $h\in \mathbb{R}^{+*}$. For all $t\in \mathbb{N}h$ we have 
$$Cov(S^h_t, I^h_t)=O_N(h).$$
\noindent Moreover, if Condition \eqref{growth} is verified, the function $O_N(\cdot)$ can be chosen such that $\displaystyle\frac{O_N(\cdot)}{N^2}$ converges pointwise towards $0$ when $\displaystyle N\rightarrow +\infty$.
\end{coro}

\begin{proof}
 Set $t=nh$ for some $n\in\mathbb{N}$ and $h\in \mathbb{R}^{+*}$. We have to proof that $\displaystyle \mathbb{E}(S^h_t I^h_t)=\mathbb{E}\left(S^h_t\right)\mathbb{E}\left(I^h_t\right)+O_N(h)$. This follows from Proposition \ref{justify_forecast_3} that gives 
\begin{equation}
\begin{array}{r c l}
{-\lambda h}\mathbb{E}(S^h_{t} I^h_{t})&=&\mathbb{E}\left(S^h_{t+h}\right)-\mathbb{E}\left(S^h_{t}\right)+O_N(h^2)\\
& &\\
&=&\text{forecast}^{1}\left(\mathbb{E}\left(S^h_t\right), \mathbb{E}\left(I^h_t\right)\right)-\mathbb{E}\left(I^h_t\right)+O_N(h^2)\\
& & \\
&=&{-\lambda h}\mathbb{E}\left(S^h_t\right)\mathbb{E}\left(I^h_t\right)+O_N(h^2).
\end{array}
\end{equation}
It results from \eqref{growth} that $O_N(\cdot)$ can be chosen such that $\displaystyle\frac{O_N(\cdot)}{N^2}$ converges pointwise towards $0$ when $\displaystyle N\rightarrow +\infty$.
\end{proof}

We have also the following corollary, saying that the expectations of $(I^h_{t})$ and  $(S^h_{t})$ do not change much by rescaling $h$, as long as $h$ is not too large. 

\begin{coro}
\label{ind_h}For all $k\in \mathbb{N}^*$ and $t\in \mathbb{N}\frac{h}{k}$ we have
\begin{equation}\label{hk}\mathbb{E}(S_t^h)=\mathbb{E}(S_{t}^{h/k})+O(h^2) \text{ and } \mathbb{E}(I_t^h)=\mathbb{E}(I_{t}^{h/k})+O(h^2),\end{equation}
where the function $O(\cdot)$ is independent of $t$ and $k$.
\end{coro}

\begin{proof} 

\noindent First let us prove the rescaling formula for $0\leq x,y\leq N$:
\begin{equation}\label{rescle}\text{forecast}_{\lambda,h}(x,y)=\text{forecast}^{(k)}_{\lambda,\frac{h}{k}}(x,y)+\begin{pmatrix}O(h^2)\\O(h^2)\end{pmatrix},\end{equation}
\noindent  where the function $O(\cdot)$ is independent of $k$. It is obtained using the fact that $l\in\mathbb{R}^+$,
\begin{equation}\label{iterate}\displaystyle \text{forecast}_{\lambda,l}(x,y)=N_{\lambda,l}(x)\begin{pmatrix} x\\y \end{pmatrix}=\begin{pmatrix}
x-\lambda y x l+x O\left(l^2\right) \\ y+y(\lambda x -\gamma)l+y O\left(l^2\right)
\end{pmatrix},\end{equation}

\noindent which leads, provided that $h>0$ is sufficiently small, for all $k\in\mathbb{N}$ and $r\leq k$, to 

$$\displaystyle \begin{pmatrix}x\left(1-\lambda N  \frac{h}{k}+ O\left(\frac{h^2}{k^2}\right)\right)^k\\ y+O(h^2)\end{pmatrix}\leq \text{forecast}^{(r)}_{\lambda,\frac{h}{k}}(x,y)\leq
\begin{pmatrix}
x+O(h^2)\\
y\left(1+(\lambda N-\gamma) \frac{h}{k}+ O\left(\frac{h^2}{k^2}\right)\right)^k
\end{pmatrix},$$  

\noindent hence

$$\displaystyle \begin{pmatrix}x \left(e^{\lambda N h+O(h^2)}+O(h^2)\right)\\ y+O(h^2)\end{pmatrix}\leq \text{forecast}^{(r)}_{\lambda,\frac{h}{k}}(x,y)
\leq 
\begin{pmatrix}
x +O(h^2)\\
y e^{(\lambda N - \gamma) h+O(h^2)}
\end{pmatrix}.$$  

\noindent Thus for all $k\in\mathbb{N}$, \begin{equation}\label{reinser}\displaystyle-\begin{pmatrix}O(h)\\ O(h)\end{pmatrix}\leq \begin{pmatrix}x\\y\end{pmatrix}-\text{forecast}^{(k)}_{\lambda,\frac{h}{k}}(x,y)\leq \begin{pmatrix}O(h)\\ O(h)\end{pmatrix}\end{equation}for some function $O(\cdot)$ independent of $k,$ $x$ and $y$, which reinserted into \eqref{iterate} give by a straightforward recurrence argument

\begin{equation}\label{recur}\displaystyle \text{forecast}^{(k)}_{\lambda,\frac{h}{k}}(x,y)=\begin{pmatrix}
x-\lambda  y x h + O\left(h^2\right) \\ y+ y(\lambda x -\gamma)h+ O\left(h^2\right)
\end{pmatrix}.\end{equation}

\noindent 
Let us just write the main argument of this recurrence. For $r\leq k$, \eqref{iterate} with $l=\frac{h}{k}$, together with \eqref{reinser}, leads to $$\displaystyle \begin{array}{r c l}\text{forecast}^{(r)}_{\lambda,\frac{h}{k}}\circ \text{forecast}_{\lambda,\frac{h}{k}} (x,y)&=&\begin{pmatrix}x-\lambda  y x \frac{h}{k}-r\lambda\left(x+O\left(h\right)\right)\left(y+O\left(h\right)\right)\frac{h}{k} + (r+1)O\left(\frac{h^2}{k^2}\right) \\ y+ y(\lambda x -\gamma)\frac{h}{k}+ r\left(y+O\left(h\right)\right)\left(\lambda\left(x+O\left(\frac{h}{k}\right)\right)-\gamma\right)\frac{h}{k} + (r+1) O\left(\frac{h^2}{k^2}\right)\end{pmatrix}\\
& &\\
 &=&\begin{pmatrix}
x-\lambda  (r+1)y x \frac{h}{k} + (r+1)O\left(\frac{h^2}{k^2}\right) \\ y+ y(\lambda x -\gamma)\frac{h}{k}+ (r+1)O\left(\frac{h^2}{k^2}\right)
\end{pmatrix}.\end{array}$$

\noindent Letting $r=k-1$, we obtain \eqref{recur}. This formula by comparison with \eqref{iterate} implies \eqref{rescle}. Hence, thanks to Proposition \ref{justify_forecast_3}, $$\displaystyle \mathbb{E}\left(I^h_{t+h}\right)=\mathbb{E}\left(I^{\frac{h}{k}}_{t+h}\right)+O\left(h^2\right)$$

\noindent for all $t\in \mathbb{N}h$, that prove the first equality in the statement. The second equality in the statement comes by exactly the same way.

\end{proof}

\noindent One consequence of this corollary is the following:

\begin{coro}\label{limit}
\noindent $\forall t\in \mathbb{N}h$, the limits $\displaystyle\lim_{l\rightarrow 0^+}\mathbb{E}(I^l_t)$ and $\displaystyle\lim_{l\rightarrow 0^+}\mathbb{E}(S^l_t)$ exist, and the convergence is uniform.\\
Moreover $$\displaystyle\lim_{l\rightarrow 0^+}\mathbb{E}(I^l_t)=\mathbb{E}(I_t^h) + O(h^2)\text{ and } \displaystyle\lim_{l\rightarrow 0^+}\mathbb{E}(S^l_t)=\mathbb{E}(S_t^h) + O(h^2),$$ where the function $O(\cdot)$ is independent of $t$.\\
\end{coro}

\begin{proof}
Taking $k\rightarrow +\infty$ in \eqref{hk}, we obtain \begin{equation}\label{limsup}\displaystyle \limsup_{l\rightarrow 0^+}{\mathbb{E}(I^l_t)}=\mathbb{E}\left(I^h_t\right)+O(h^2),\end{equation}
\noindent Hence

$$\displaystyle \limsup_{l\rightarrow 0^+}{\mathbb{E}(I^l_t)}=\liminf_{h\rightarrow 0^+}{\mathbb{E}\left(I^h_t\right)}.$$

\noindent Thus the limit superior in \eqref{limsup} can be replaced by a limit. Moreover, the same relations holds for $\mathbb{E}(S^l_t)$ instead of $\mathbb{E}(I^l_t)$. 
\end{proof}






\end{appendices}


\bibliographystyle{acm}	
\bibliography{bibliography.bib}

\end{document}